\documentclass{amsart}
\usepackage{amsmath, amssymb}
\usepackage[utf8]{inputenc}
\usepackage[english]{babel}
\usepackage{geometry}
\usepackage{color,graphicx}

\usepackage{url}
\usepackage[colorlinks = true,
linkcolor=blue,
urlcolor=blue,
citecolor=blue,
anchorcolor=blue,
backref=page]{hyperref}
\usepackage{tikz}

\newtheorem{theorem}{Theorem}[section]
\newtheorem{lemma}{Lemma}[section]

\newtheorem{example}{Example}[section]
\newtheorem{remark}{Remark}[section]
\newtheorem{question}{Question}[section]
\newtheorem{conjecture}{Conjecture}[section]

\newcommand{\vol}{\mathrm{vol}}
\newcommand{\tomega}{\widetilde{\omega}}
\newcommand{\tN}{\widetilde{\ver}}
\newcommand{\voct}{v_{\mathrm{oct}}}
\newcommand{\vtet}{v_{\mathrm{tet}}}
\newcommand{\ver}{\mathrm{ver}}

\begin{document}
	
\title[Spectra of normalized volumes]{Spectra of normalized volumes \\ of right-angled hyperbolic polyhedra}

\author[A.~Egorov, A.~Vesnin]{A.~Egorov, A.~Vesnin} 
\address{Sobolev Institute of Mathematics, Novosibirsk, Russia} 
\email{a.egorov2@g.nsu.ru} 

\address{Sobolev Institute of Mathematics, Novosibirsk, Russia}
\email{vesnin@math.nsc.ru}

\begin{abstract} 
We consider three-dimensional hyperbolic polyhedra of finite volume with finitely many
vertices. The normalized volume of a polyhedron is the ratio of its volume to the number of vertices.
Given some set of hyperbolic polyhedra, we can associate with it the set of normalized volumes of
the polyhedra belonging to it. We call this set the spectrum of normalized volumes of the set under
consideration. We focus on right-angled hyperbolic polyhedra. For the subset of ideal polyhedra,
we find bounds for the spectrum of normalized volumes and prove that they are sharp. Moreover,
we show that the spectrum splits into discrete and dense parts. For the subset of compact polyhedra,
we obtain estimates for the spectrum of normalized volumes, prove that the upper bound is sharp,
and also establish numerical intervals on which the spectrum is discrete and dense. The endpoints of
the indicated numerical intervals are expressed in terms of the volume of the regular ideal hyperbolic
tetrahedron and the volume of the regular ideal hyperbolic octahedron.
\end{abstract}




\thanks{The authors were supported by the state task of Sobolev Institute of Mathematics, namely A.\,Yu.~Vesnin was supported by project No.~FWNF-2026-0031, and A.\,A.~Egorov was supported by project No.~FWNF-2026-0011. \\ This is a preprint of the Work accepted for publication in \textit{Siberian Mathematical Journal}, \textcopyright\ 2026, Pleiades Publishing, Ltd. An electronic link to the Publisher's website is available at: \url{http://pleiades.online/}.}

\maketitle 
	
\section{Introduction} \label{sec1}
	
The study of volumes of three-dimensional hyperbolic manifolds and hyperbolic polyhedra is an important direction in the development of modern geometry. It is of interest to describe sets of volumes and their properties for various classes of polyhedra (see, for instance,~\cite{Bel21, CGV}). On the one hand, we can fix the combinatorial structure of a polyhedron and ask for which dihedral angles the polyhedron (with finite or ideal vertices) has maximal volume. For example, by a well-known result of Milnor~\cite{Mil82}, the greatest volume of a hyperbolic $3$-simplex is attained on the regular ideal $3$-simplex, i.e., on a simplex all of whose vertices are ideal and all of whose dihedral angles are equal to one another and hence are equal to $\pi/3$. An analogous result holds in the $n$-dimensional case. As shown in~\cite{HaMu81}, the greatest volume of a hyperbolic $n$-simplex is attained on the regular ideal $n$-simplex. On the other hand, one can fix the dihedral angles and ask which polyhedron has maximal volume for a given number of faces or vertices. A hyperbolic polyhedron is called right-angled if all of its dihedral angles are equal to $\pi/2$. For the construction of three-dimensional hyperbolic manifolds corresponding to torsion-free subgroups of reflection groups of right-angled polyhedra, see~\cite{Ves17}. The aim of the present paper is to study volumes of right-angled polyhedra normalized by the number of vertices.

Following~\cite{Ves10}, for a hyperbolic polyhedron $\mathcal P$ of finite volume $\vol(\mathcal P)$ with finitely many vertices $\ver(\mathcal P)$, we define the \textit{normalized volume} as the value $\omega(\mathcal P) = \vol(\mathcal P) / \ver(\mathcal P)$. Let $\mathcal R$ be some set of hyperbolic polyhedra. The set $\Omega (\mathcal R) = \{ \omega(\mathcal P) \mid \mathcal P \in \mathcal R \}$ of normalized volumes of polyhedra from $\mathcal R$ will be called the \textit{spectrum of normalized volumes} of $\mathcal R$.

Two-sided estimates for the volumes of right-angled hyperbolic polyhedra in terms of the number of their faces were obtained by Atkinson in~\cite{At09}. A list of the $825$ compact right-angled hyperbolic polyhedra of smallest volume was given by Inoue in~\cite{In22}. A list of the $248$ smallest volume values of ideal right-angled hyperbolic polyhedra was given in~\cite{VeEg}.

Denote by $\mathcal R_{ideal}$ the set of all ideal right-angled hyperbolic polyhedra and by $\mathcal R_{comp}$ the set of all compact right-angled hyperbolic polyhedra. In both cases, we assume that the polyhedra have finite volume and finitely many vertices.

The normalized volumes of these sets and their spectra are described in the following theorems, where the constants used have the following values to six decimal places: 
$\voct=3.663862$ and $\vtet = 1.014941$. 
	
\begin{theorem} \label{theorem1}
The spectrum of normalized volumes of ideal right-angled polyhedra $\Omega (\mathcal R_{ideal}) = \{\omega(\mathcal P) \mid \mathcal P \in \mathcal{R}_{ideal}\}$ belongs to the interval $\left[ \frac{1}{6} \voct, \frac{1}{2} \voct \right]$, where the lower bound is attained on the octahedron and the upper bound is sharp. The spectrum is discrete on the interval $\left[ \frac{1}{6} \voct, \frac{1}{4} \voct \right)$ and dense on the interval $\left[ \frac{1}{4} \voct, \frac{1}{2} \voct \right]$.
\end{theorem}
	
\begin{theorem} \label{theorem2}
The spectrum of normalized volumes of compact right-angled polyhedra $\Omega (\mathcal R_{comp}) =  \{\omega(\mathcal P) \mid \mathcal P \in \mathcal{R}_{comp}\}$ belongs to the interval $\left[ \frac{5}{192} \voct,  \frac{5}{8} \vtet \right)$, where the upper bound is sharp. On the interval $\left[ \frac{5}{192} \voct, \frac{1}{32} \voct \right)$ the spectrum is discrete, while on the interval $\left[ \frac{5}{16} \vtet, \frac{5}{8} \vtet \right]$, it is dense.
\end{theorem} 
		
The behavior of the spectrum $\Omega (\mathcal R_{comp})$ on the interval $\left[ \frac{1}{32} \voct, \frac{5}{16} \vtet \right]$ remains unclear; see Question~\ref{question3}.		

Note the relation between the obtained results and the theory of hyperbolic knots and links. For a hyperbolic link $L \subset S^3$, consider the volume of its complement endowed with the complete metric of constant negative curvature and denote it by $\vol(L)$. Let $c(L)$ be the minimum number of crossings among all diagrams of the link $L$. In~\cite{CKP2016}, Champanerkar, Kofman, and Purcell defined the \textit{volume density} of a link as follows: $\operatorname{vd} (L) = \vol(L) / c(L)$. It is known that $\operatorname{vd} (L) \leq \voct$. A sequence of links $L_n$ such that $c(L_n) \to \infty$ is called \textit{geometrically maximal} if $\lim_{n\to \infty} \operatorname{vd} (L_n) = \voct$. Examples of geometrically maximal sequences of links were constructed in~\cite{CKP2016}.

Subsequently, Burton~\cite{Bur2016} showed that the set of volume densities of all hyperbolic links is dense in the interval $\left[ 0, \voct \right]$. Adams, Calderon, and Mayer showed in~\cite{ACM} that the set of volume densities of hyperbolic links in thickened surfaces $S_g \times I$, where $g \geq 2$, is dense in the interval $\left[0, 2\voct \right]$.

In~\cite{KwonTham}, Kwon and Tham studied a special class of hyperbolic links called fully augmented links, or FAL for short. For them, the natural measure of complexity is the number of augmentations, which are also called the vertical components of a link. For this class, Kwon and Tham defined volume density as the ratio of the volume of the link complement to the number of its vertical components and showed that the set of densities lies in the interval $\left[ \voct,10\vtet \right]$, with both bounds sharp. Moreover, they established that the spectrum is discrete in $\left[ \voct, 2\voct \right)$ and dense in $\left[ 2\voct, 10\vtet \right]$.
	
There is a natural geometric relation between fully augmented links and ideal right-angled polyhedra: the complement of every such link decomposes into two isometric ideal right-angled polyhedra~\cite{La04,Pur2011,VeEg24} (for an analogous decomposition of complements of links in the thickened torus, see~\cite{Kw}). In this case, the volume of the link is equal to twice the volume of the polyhedron, while the number of vertical components of the link is related to the number of vertices of the polyhedron. Consequently, using a suitable normalization, one can embed the spectrum of volume densities of links into the spectrum of normalized volumes of polyhedra. However, the class of ideal right-angled polyhedra is substantially broader than the subclass of polyhedra arising from augmented links, and therefore the question of describing the volume spectrum of the class of all ideal right-angled hyperbolic polyhedra naturally arises.
	
The paper is organized as follows.
In Section~\ref{prelim-ideal}, we give the main definitions and properties of ideal right-angled hyperbolic polyhedra. Section~\ref{prelim-compact} is devoted to compact right-angled hyperbolic polyhedra, in particular, to L\"obell polyhedra. In Sections~\ref{main-ideal} and~\ref{main-compact}, we prove Theorems~\ref{theorem1} and~\ref{theorem2}, respectively. In Section~\ref{comparison}, we compare the obtained results on normalized volumes of ideal and compact right-angled polyhedra with the results of Kwon and Tham on the volume density spectrum of augmented links. In conclusion, we formulate some open questions and one conjecture.
	
\section{Ideal right-angled polyhedra} \label{prelim-ideal}
	
An \textit{abstract polyhedron} is a cell complex on $S^2$ that can be realized as a convex polyhedron in three-dimensional Euclidean space. By the classical theorem of Steinitz (see, for instance,~\cite[Ch.~1]{Pra}), realizability of a cell complex as a convex Euclidean polyhedron is equivalent to the condition that its one-dimensional skeleton is a $3$-connected graph without loops and multiple edges. Recall that a graph is called $k$-connected if it contains at least $k+1$ vertices and any two of its distinct vertices can be joined by at least $k$ independent paths. By the Menger--Whitney theorem, a graph containing at least $k+1$ vertices is $k$-connected if and only if the deletion of any $k-1$ of its vertices (and the edges incident to them) results in a connected graph. Denote by $\operatorname{E} (P)$ the set of edges of an abstract polyhedron $P$. A \textit{labeling} of an abstract graph $P$ is a mapping $\mathcal E : \operatorname{E}(P) \to (0,\pi/2]$. The pair $(P, \mathcal E)$ is called a \textit{labeled abstract polyhedron}. We will say that a labeled abstract polyhedron is \textit{realized as a hyperbolic polyhedron} if there exists a polyhedron $\mathcal P$ in hyperbolic space $\mathbb H^3$ such that there is a label-preserving graph isomorphism taking the one-dimensional skeleton $\mathcal P^{(1)}$, with its edges labeled by the dihedral angles, to the labeled abstract polyhedron $(P, \mathcal E)$. Below we consider polyhedra in $\mathbb H^3$ all of whose dihedral angles are equal to $\pi/2$ and call them \textit{right-angled}. A vertex of a polyhedron in $\mathbb H^3$ is called \textit{ideal} if it lies on the absolute $\partial \mathbb H^3$. For more details on polyhedra in the space $\mathbb H^3$, see~\cite{AVS}.

Let $G$ be a planar graph and let $G^*$ be the graph dual to $G$. A simple closed curve composed of $k \geq 3$ edges in $G^*$ is called a \textit{$k$-circuit}. A \textit{prismatic $k$-circuit} is a $k$-circuit $\gamma$ such that no two edges of the graph $G$ crossing $\gamma$ have a common vertex.

Necessary and sufficient conditions for the realization in $\mathbb H^3$ of a labeled abstract polyhedron of finite volume were obtained in the works of E.~M.~Andreev~\cite{Andreev1, Andreev2} (see also~\cite{RHD}). We state the result for the special case of right-angled polyhedra, following~\cite{At09}. In this case, the necessary and sufficient conditions are purely combinatorial.
	
\begin{theorem} \label{theorem2.1}
A right-angled abstract polyhedron $(P, \mathcal E)$ is realized as a hyperbolic polyhedron, denoted by $\mathcal P$, if and only if the following conditions hold:
\begin{itemize}
\item[(1)] $P$ has at least six faces;
\item[(2)] each vertex in $P$ has valency $3$ or $4$; 
\item[(3)] for all triples of faces $(F_i, F_j, F_k)$ in $P$ such that $F_i \cap F_j$ and $F_j \cap F_k$ are edges in $P$ with distinct endpoints, we have $F_i \cap F_k = \emptyset$;
\item[(4)] $P^*$ has no prismatic $4$-circuits. 
\end{itemize}
Moreover, every 3-valent vertex in $P$ corresponds to a finite vertex in $\mathcal P$, every 4-valent vertex in $P$ corresponds to an ideal vertex in $\mathcal P$, and the realization is unique up to isometry.
\end{theorem} 
	
An \textit{ideal right-angled polyhedron} is a right-angled hyperbolic polyhedron all of whose vertices are ideal. Ideal right-angled polyhedra are naturally related to knot theory, since the complements of hyperbolic fully augmented links decompose into pairs of isometric ideal right-angled polyhedra (see~\cite{La04, Pur2011,VeEg24}).
	
For a polyhedron $\mathcal P \in \mathcal R_{ideal}$, denote by $F$ the number of faces, by $V$ the number of vertices, and by $p_k$ the number of $k$-gonal faces, where $k \geq 3$. Then $F = \sum_{k\geq 3} p_k$. Since all vertices of the polyhedron are 4-valent, Euler's formula for polyhedra implies that $F = V + 2$, whence
\begin{equation} \label{eqn1}
p_3 = 8 + \sum_{k \geq 5} (k-4) p_k. 
\end{equation} 	
In particular, the ideal right-angled polyhedron with the smallest number of faces is the octahedron. 	
	
In the following example, we describe an infinite family of ideal right-angled polyhedra generalizing the octahedron, which will play an important role in our considerations.
	
\begin{example} \label{ex:antiprism}
{\rm 
For $n \ge 3$, an \textit{$n$-antiprism} $A(n)$ is a $(2n+2)$-hedron with upper and lower $n$-gonal bases and with lateral surface formed by two layers of $n$ triangles each, in which four edges meet at every vertex. The Schlegel diagrams of the antiprisms $A(3)$ and $A(4)$ are shown in Fig.~\ref{fig1}. In particular, $A(3)$ is the octahedron. Realizability of right-angled antiprisms as ideal hyperbolic polyhedra follows from Theorem~\ref{theorem2.1}. A right-angled hyperbolic $n$-antiprism will be denoted by $\mathcal A(n)$.
}
\end{example}
\begin{figure}[ht]
\begin{center}
\unitlength=.1mm
\begin{tikzpicture}[scale=0.4] 
\unitlength=10.mm
  \fill[black]  (-16,6) circle (5pt) ;
  \fill[black]  (-15,4) circle (5pt) ;
  \fill[black]  (-17,4) circle (5pt) ;
  \fill[black]  (-16,1) circle (5pt) ;
  \fill[black]  (-20,7) circle (5pt) ;
  \fill[black]  (-12,7) circle (5pt) ;
\draw [very thick, black] (-20,7)-- (-12,7) -- (-16,1) -- cycle;
\draw [very thick, black] (-15,4)-- (-16,6) -- (-17,4) -- cycle;
\draw [very thick, black] (-16,1)-- (-17,4) -- (-20,7) -- (-16,6) -- (-12,7) -- (-15,4) -- cycle;
\draw(-16,-1) node {$A(3)$}; 
  \fill[black]  (0,5) circle (5pt) ;
  \fill[black]  (-8,5) circle (5pt) ;
  \fill[black]  (-4,1) circle (5pt) ;
  \fill[black]  (-4,9) circle (5pt) ;
\draw [very thick, black] (0,5)-- (-4,9) -- (-8,5) -- (-4,1) -- cycle;
  \fill[black]  (-3,4) circle (5pt) ;
  \fill[black]  (-5,4) circle (5pt) ;
  \fill[black]  (-3,6) circle (5pt) ;
  \fill[black]  (-5,6) circle (5pt) ;
\draw [very thick, black] (-3,4)-- (-3,6) -- (-5,6) -- (-5,4) -- cycle;
\draw [very thick, black] (-3,4)-- (0,5) -- (-3,6) -- (-4,9) --(-5,6) -- (-8,5) -- (-5,4) -- (-4,1) -- cycle;
\draw(-4,-1) node {$A(4)$};
  \fill[black]  (4,5) circle (5pt) ;
  \fill[black]  (12,5) circle (5pt) ;
  \fill[black]  (8,1) circle (5pt) ;
  \fill[black]  (8,9) circle (5pt) ;
\draw [very thick, black] (4,5)-- (8,9) -- (12,5) -- (8,1) -- cycle;
\draw [very thick, red] (8,5) circle[radius=0.1cm];
\draw [very thick, red] (7,4)-- (8,5) -- (7,6);
\draw [very thick, red] (9,4)-- (8,5) -- (9,6);
\draw [very thick, black]  (7,6) -- (9,6);
\draw [very thick, black] (9,4) -- (7,4);
\draw [very thick, black] (7,4)-- (4,5) -- (7,6) -- (8,9) --(9,6) -- (12,5) -- (9,4) -- (8,1) -- cycle;
  \fill[black]  (7,4) circle (5pt) ;
  \fill[black]  (9,4) circle (5pt) ;
  \fill[black]  (7,6) circle (5pt) ;
  \fill[black]  (9,6) circle (5pt) ;
\draw(8,-1) node {$A(4)^*$};
\end{tikzpicture}
\end{center}
\caption{Antiprisms $A(3)$ and $A(4)$, twisted antiprism $A(4)^*$.} \label{fig1}
\end{figure}
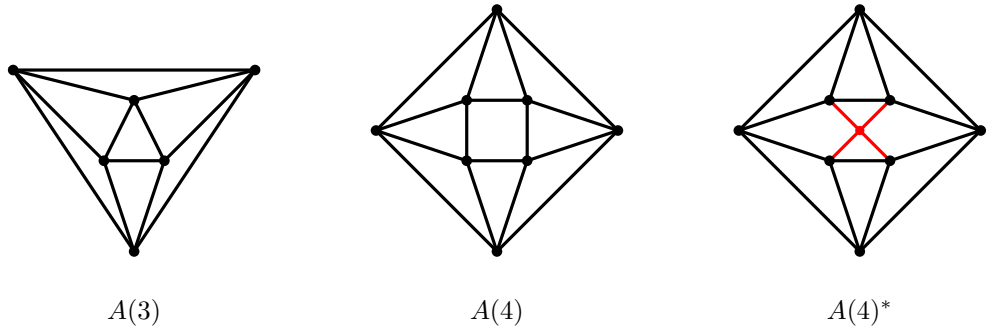
		
Following~\cite{Bri}, we define an operation called the \textit{edge-twist}. The operation is as follows. Let $P$ be an abstract polyhedron having in some face four distinct 4-valent vertices forming two pairs of adjacent vertices joined by edges $e_1$ and $e_2$. Apply a transformation in the face consisting of deleting the edges $e_{1}$ and $e_{2}$, creating a new vertex $v$, and joining this vertex by edges to the four vertices under consideration. The polyhedron obtained as a result of this transformation is denoted by $P^{*}$. We will say that $P^{*}$ is obtained from $P$ by an \textit{edge-twist}. In Fig.~\ref{fig2}, we indicate the edges $e_{1}$ and $e_{2}$ involved in the twisting and the new vertex $v$ with the edges incident to it.
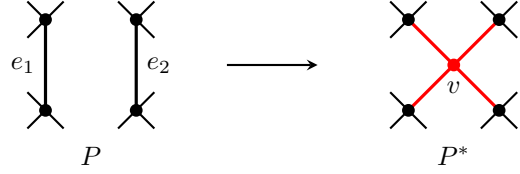
\begin{figure}[ht]
\centering
\setlength{\unitlength}{1.0mm}
 \begin{tikzpicture}[scale=1.2] 
\coordinate (1) at (0, 0) ;
  \coordinate (2) at (1, 0) ;
  \coordinate (3) at (0, 1) ;
  \coordinate (4) at (1, 1) ;
  \draw[very thick, black] (1) -- (3) ;
  \draw[very thick, black] (2) -- (4) ; 
    \draw[thick](0,1) -- (0.2,1.2);
    \draw[thick](0,1) -- (-0.2,1.2);
    \draw[thick](0,1) -- (-0.2,0.8);
    \draw[thick](1,1) -- (0.8,1.2);
    \draw[thick](1,1) -- (1.2,1.2);
    \draw[thick](1,1) -- (1.2,0.8);
    \draw[thick](1,0) -- (1.2,0.2);
    \draw[thick](1,0) -- (1.2,-0.2);
    \draw[thick](1,0) -- (0.8,-0.2);
    \draw[thick](0,0) -- (0.2,-0.2);
    \draw[thick](0,0) -- (-0.2,-0.2);
    \draw[thick](0,0) -- (-0.2,0.2);
  \fill[black] (1) circle (2pt) ;
  \fill[black] (2) circle (2pt) ; 
  \fill[black] (3) circle (2pt) ;
  \fill[black] (4) circle (2pt) ;  
  \draw[-stealth,thick] (2,0.5) -- (3, 0.5);
\coordinate (6) at (4, 0) ;
  \coordinate (7) at (5, 0) ;
  \coordinate (8) at (4, 1) ;
  \coordinate (9) at (5, 1) ;
  \coordinate (5) at (4.5, 1/2) ;
  \draw[very thick, red] (6) -- (5) ;
  \draw[very thick, red] (5) -- (7) ;
  \draw[very thick, red] (8) -- (5) ; 
  \draw[very thick, red] (5) -- (9) ; 
  \fill[black] (6) circle (2pt) ;
  \fill[black] (7) circle (2pt) ; 
  \fill[black] (8) circle (2pt) ;
  \fill[black] (9) circle (2pt) ;
  \fill[red] (5) circle (2pt) ; 
      \draw[thick](4,1) -- (4.2,1.2);
    \draw[thick](4,1) -- (3.8,1.2);
    \draw[thick](4,1) -- (3.8,0.8);
    \draw[thick](5,1) -- (4.8,1.2);
    \draw[thick](5,1) -- (5.2,1.2);
    \draw[thick](5,1) -- (5.2,0.8);
    \draw[thick](5,0) -- (5.2,0.2);
    \draw[thick](5,0) -- (5.2,-0.2);
    \draw[thick](5,0) -- (4.8,-0.2);
    \draw[thick](4,0) -- (4.2,-0.2);
    \draw[thick](4,0) -- (3.8,-0.2);
    \draw[thick](4,0) -- (3.8,0.2); 
    \draw(4.5,0.25) node {$v$};
     \draw(-0.25,0.5) node {$e_1$};
      \draw(1.25,0.5) node {$e_2$};
         \draw(0.5,-0.5) node {$P$};
           \draw(4.5,-0.5) node {$P^*$};
    \end{tikzpicture}
    \caption{Edge-twist of the edges $e_1$ and $e_2$.}
    \label{fig2} 
\end{figure}

It follows from Theorem~\ref{theorem2.1} that if $\mathcal P \in \mathcal R_{ideal}$, then $\mathcal P^* \in \mathcal R_{ideal}$. The following result shows that every polyhedron from $\mathcal{R}_{ideal}$ is an antiprism or can be obtained from an antiprism by finitely many edge-twist operations.
	
\begin{theorem} \cite{Bri} \label{theoremBri}  \label{theorem2.2} 
The class of polyhedra $\mathcal{R}_{ideal}$ is generated by the antiprisms $\mathcal A(n)$, $n \geq 3$, and edge-twist operations. 
\end{theorem}

Traditionally, volumes of polyhedra in three-dimensional hyperbolic space are computed using the \textit{Lobachevsky function} $\Lambda(\theta)$ introduced by Milnor~\cite{Mil82},
$$
\Lambda(\theta) = -\int_0^\theta \log |2\sin t| dt. 
$$ 
As shown by Thurston~\cite{Th80}, the volume of the right-angled hyperbolic $n$-antiprism $\mathcal A(n)$, $n \geq 3$, is expressed by the following formula:
\begin{equation} \label{eqn2}
\vol(\mathcal A(n)) = 2n \left[ \Lambda\left( \frac{\pi}{4} + \frac{\pi}{2n} \right) + \Lambda\left( \frac{\pi}{4} - \frac{\pi}{2n} \right) \right].
\end{equation} 
Below, all approximate values of the Lobachevsky function and volumes of polyhedra will be given to six decimal places, without indicating the omitted digits by an ellipsis.
In particular, for $n=3$ we obtain the volume of the ideal right-angled octahedron:
$$
\voct = \vol(\mathcal A(3)) = 8\Lambda(\pi/4) = 3.663862.
$$	
It follows from formula~\eqref{eqn2} that $\displaystyle \lim_{n\to\infty}\frac{\vol(\mathcal A(n))}{\ver(\mathcal A(n))}=\frac{\voct}{2}$.

In Fig.~\ref{fig3}, we show a graph representing five generations of polyhedra obtained from $\mathcal A(4)$ by edge-twist operations. The vertices of the graph correspond to the polyhedra $\mathcal P_1, \ldots, \mathcal P_{24}$, where $\mathcal P_1= \mathcal A(4)$. Two vertices are joined by an edge if the polyhedra corresponding to these vertices are related by one edge-twist transformation. 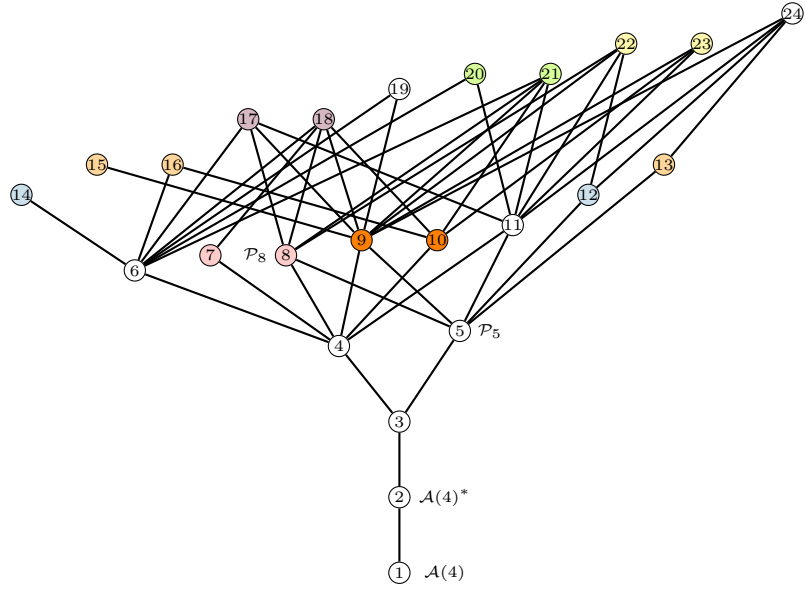
\begin{figure}[ht]
	\centering
\begin{tikzpicture}[main_node/.style={circle,draw,minimum size=0.8em, inner sep=1.2pt]}] 
\node[main_node] (1) at (0, 0) {};
\node[main_node] (2) at (0, 1) {};
\node[main_node] (3) at (0, 2) {};
\node[main_node] (4) at (-0.8, 3) {};
\node[main_node] (5) at (0.8, 3.2) {};
\node[main_node] (6) at (-3.5, 4) {};
\node[main_node, fill=red!20] (7) at (-2.5, 4.2) {};
\node[main_node, fill=red!20] (8) at (-1.5, 4.2) {};
\node[main_node, fill=orange] (9) at (-0.5, 4.4) {};
\node[main_node, fill=orange] (10) at (0.5, 4.4) {};
\node[main_node] (11) at (1.5, 4.6) {};
\node[main_node, fill=green!40!blue!20] (12) at (2.5, 5.0) {};
\node[main_node, fill=red!40!yellow!40] (13) at (3.5, 5.4) {};
\node[main_node, fill=green!40!blue!20] (14) at (-5,5.0){};
\node[main_node, fill=red!40!yellow!40] (15) at (-4, 5.4) {};
\node[main_node, fill=red!40!yellow!40] (16) at (-3, 5.4) {};
\node[main_node, fill=blue!40!orange!40] (17) at (-2, 6.0) {};
\node[main_node, fill=blue!40!orange!40] (18) at (-1, 6.0) {};
\node[main_node] (19) at (0, 6.4){};
\node[main_node, fill=green!40!yellow!40] (20) at (1, 6.6) {};
\node[main_node, fill=green!40!yellow!40] (21) at (2, 6.6) {};
\node[main_node, fill=yellow!40!yellow!40] (22) at (3, 7.0) {};
\node[main_node, fill=yellow!40!yellow!40] (23) at (4, 7.0) {};
\node[main_node] (24) at (5.2, 7.4) {};
 \path[draw, thick]
(1) edge node {} (2)
(2) edge node {} (3)
(3) edge node {} (4)
(3) edge node {} (5)
(4) edge node {} (6)
(4) edge node {} (7)
(4) edge node {} (8)
(4) edge node {} (9)
(4) edge node {} (10)
(4) edge node {} (11)
(5) edge node {} (8)
(5) edge node {} (9)
(5) edge node {} (11)
(5) edge node {} (12)
(5) edge node {} (13)
(6) edge node {} (14)
(6) edge node {} (16)
(6) edge node {} (17)
(6) edge node {} (18)
(6) edge node {} (19)
(6) edge node {} (20)
(6) edge node {} (21)
(7) edge node {} (18)
(8) edge node {} (17)
(8) edge node {} (18)
(8) edge node {} (21)
(8) edge node {} (22)
(9) edge node {} (15)
(9) edge node {} (17)
(9) edge node {} (18)
(9) edge node {} (19)
(9) edge node {} (21)
(9) edge node {} (22)
(9) edge node {} (23)
(9) edge node {} (24)
(10) edge node {} (16)
(10) edge node {} (18)
(10) edge node {} (21)
(10) edge node {} (23)
(11) edge node {} (17)
(11) edge node {} (20)
(11) edge node {} (21)
(11) edge node {} (22)
(11) edge node {} (23)
(11) edge node {} (24)
(12) edge node {} (22)
(12) edge node {} (24)
(13) edge node {} (24);  
\draw(0,0) node {\tiny{1}};
\draw(0,1) node {\tiny{2}};
\draw(0,2) node {\tiny{3}};
\draw(-0.8,3) node {\tiny{4}};
\draw(0.8,3.2) node {\tiny{5}};
\draw(-3.5,4) node {\tiny{6}};
\draw(-2.5,4.2) node {\tiny{7}};
 \draw(-1.5,4.2) node {\tiny{8}};
 \draw(-0.5,4.4) node {\tiny{9}};
\draw(0.5,4.4) node {\tiny{10}};
 \draw(1.5,4.6) node {\tiny{11}};
 \draw(2.5,5.0) node {\tiny{12}};
\draw(3.5,5.4) node {\tiny{13}};
 \draw(-5,5.0) node {\tiny{14}};
\draw(-4,5.4) node {\tiny{15}};
\draw(-3,5.4) node {\tiny{16}};
\draw(-2,6.0) node {\tiny{17}};
\draw(-1,6.0) node {\tiny{18}};
\draw(0,6.4) node {\tiny{19}};
\draw(1,6.6) node {\tiny{20}};
\draw(2,6.6) node {\tiny{21}};
\draw(3,7.0) node {\tiny{22}};
\draw(4,7.0) node {\tiny{23}};
\draw(5.2,7.4) node {\tiny{24}};
\draw(0.6,0) node {\tiny $\mathcal A(4)$};
\draw(0.6,1) node {\tiny $\mathcal A(4)^*$};
\draw(1.2,3.2) node {\tiny $\mathcal P_{5}$};
\draw(-1.9,4.2) node {\tiny $\mathcal P_{8}$};
\end{tikzpicture}
\caption{Initial part of the growth graph of polyhedra from $\mathcal A(4)$.}
\label{fig3}
\end{figure}
Note that, within approximate computations, the following volume equalities hold:
\begin{equation}
\begin{gathered}
\operatorname{vol} (\mathcal P_7) = \operatorname{vol} (\mathcal P_8) = 10.991587; \quad 
\operatorname{vol} (\mathcal P_9) = \operatorname{vol} (\mathcal P_{10}) = 11.136296; \cr 
\operatorname{vol} (\mathcal P_{12}) = \operatorname{vol} (\mathcal P_{14}) = 11.801747; \quad 
\operatorname{vol} (\mathcal P_{13}) = \operatorname{vol} (\mathcal P_{15}) =  \operatorname{vol} (\mathcal P_{16}) = 12.046092; \cr 
\operatorname{vol} (\mathcal P_{17}) = \operatorname{vol} (\mathcal P_{18}) = 12.276278; \quad 
\operatorname{vol} (\mathcal P_{20}) = \operatorname{vol} (\mathcal P_{21}) = 12.611908; \cr 
\operatorname{vol} (\mathcal P_{22}) = \operatorname{vol} (\mathcal P_{23}) = 12.883862. 
\end{gathered}
\end{equation} 

\begin{remark} \label{remark2.1} 
{\rm 
Under an edge-twist operation, the normalized volume can both increase and decrease. An example of an increase in the normalized volume is given by the pair $\mathcal A(4)$ and $\mathcal A(4)^*$. The inequality $\omega(\mathcal A(4)^*) > \omega(\mathcal A(4))$ holds, since $\omega(\mathcal A(4)) = \frac{6.023046}{8} = 0.75288$, while $\omega(\mathcal A(4)^*) = \frac{7.327725}{9} = 0.814191$. Since the polyhedron $\mathcal A(4)^*$ can be assembled from two copies of $\mathcal A(3)$, we have $\vol (\mathcal A(4)^*) = 2 \vol (\mathcal A(3))$. The volumes of the polyhedra $\mathcal A(3)$ and $\mathcal A(4)$ are computed by formula~(\ref{eqn2}). An example of a decrease in the normalized volume is given by the pair of polyhedra marked in Fig.~\ref{fig3} as $\mathcal P_{5}$ and $\mathcal P_8 = \mathcal P^*_5$. The Schlegel diagrams of these polyhedra are shown in Fig.~\ref{fig4}. The inequality $\omega(\mathcal P^*_5) < \omega(\mathcal P_5)$ holds, since $\omega(\mathcal P_5) = \tfrac{10.149416}{11} = 0.922674$, while $\omega(\mathcal P_5^*) = \omega(\mathcal P_8) = \frac{10.991587}{12} = 0.915965$. The volumes of the polyhedra $\mathcal P_5$ and $\mathcal P_8$ were computed in~\cite{VeEg}.
} 
\end{remark}
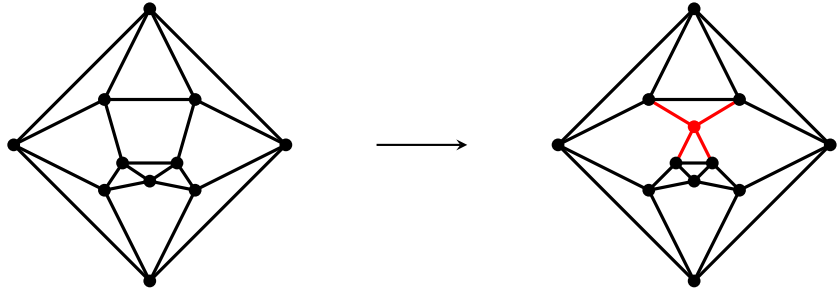
\begin{figure}[htb]
    \centering
    \setlength{\unitlength}{1.0mm}
    \begin{tikzpicture}[scale=1.2] 
\coordinate (A) at (-1.5, 1.5) ;
  \coordinate (B) at (0, 3) ;
  \coordinate (C) at (1.5, 1.5) ;
  \coordinate (D) at (0, 0) ;
  \fill[black] (A) circle (2pt) ;
  \fill[black] (B) circle (2pt) ; 
  \fill[black] (C) circle (2pt) ;
  \fill[black] (D) circle (2pt) ;
  \coordinate (E) at (-0.5, 1.) ;
  \coordinate (F) at (-0.5, 2.) ;
  \coordinate (G) at (0.5, 2.) ;
  \coordinate (H) at (0.5, 1.) ;
  \coordinate (K) at (0, 1.1) ;
  \coordinate (L) at (-0.3,1.3) ;
  \coordinate (M) at (0.3,1.3) ;
  \fill[black] (E) circle (2pt) ;
  \fill[black] (F) circle (2pt) ; 
  \fill[black] (G) circle (2pt) ;
  \fill[black] (H) circle (2pt) ;
  \fill[black] (K) circle (2pt) ; 
  \fill[black] (L) circle (2pt) ;
 \fill[black] (M) circle (2pt) ;
\draw[very thick, black] (A) -- (B) -- (C) -- (D) -- cycle;
\draw[very thick, black] (A) -- (F) -- (B) -- (G) -- (C) -- (H) -- (D) -- (E) -- cycle;
\draw[very thick, black] (E) -- (K) -- (L) -- cycle;
\draw[very thick, black] (H) -- (K) -- (M) -- cycle;
\draw[very thick, black] (L) -- (F);
\draw[very thick, black] (M) -- (G);
\draw[very thick, black] (F) -- (G);
\draw[very thick, black] (L) -- (M);
  \draw[-stealth,thick] (2.5,1.5) -- (3.5,1.5);
\coordinate (A2) at (4.5, 1.5) ;
  \coordinate (B2) at (6, 3) ;
  \coordinate (C2) at (7.5, 1.5) ;
  \coordinate (D2) at (6, 0) ;
  \fill[black] (A2) circle (2pt) ;
  \fill[black] (B2) circle (2pt) ; 
  \fill[black] (C2) circle (2pt) ;
  \fill[black] (D2) circle (2pt) ;
  \coordinate (E2) at (5.5, 1.) ;
  \coordinate (F2) at (5.5, 2.) ;
  \coordinate (G2) at (6.5, 2.) ;
  \coordinate (H2) at (6.5, 1.) ;
  \coordinate (K2) at (6, 1.1) ;
  \coordinate (L2) at (5.8,1.3) ;
  \coordinate (M2) at (6.2,1.3) ;
   \coordinate (O2) at (6,1.7) ;
  \fill[red] (O2) circle (2pt) ;
  \draw[very thick, red] (L2) -- (O2) -- (M2);
  \draw[very thick, red] (F2) -- (O2) -- (G2);
    \fill[black] (E2) circle (2pt) ;
  \fill[black] (F2) circle (2pt) ; 
  \fill[black] (G2) circle (2pt) ;
  \fill[black] (H2) circle (2pt) ;
  \fill[black] (K2) circle (2pt) ; 
  \fill[black] (L2) circle (2pt) ;
 \fill[black] (M2) circle (2pt) ;
 \draw[very thick, black] (L2) -- (M2);
\draw[very thick, black] (A2) -- (B2) -- (C2) -- (D2) -- cycle;
\draw[very thick, black] (A2) -- (F2) -- (B2) -- (G2) -- (C2) -- (H2) -- (D2) -- (E2) -- cycle;
\draw[very thick, black] (E2) -- (K2) -- (L2) -- cycle;
\draw[very thick, black] (H2) -- (K2) -- (M2) -- cycle;
\draw[very thick, black] (F2) -- (G2);
    \end{tikzpicture}
    \caption{The polyhedra $\mathcal P_5$ and $\mathcal P_8 = \mathcal P_5^*$.}
    \label{fig4}   
\end{figure}

Upper and lower estimates for volumes of ideal right-angled polyhedra in terms of the number of vertices were obtained by Atkinson in~\cite{At09}.

\begin{theorem} \cite[Theorem 2.2]{At09}  \label{theorem2.3}
If an ideal right-angled hyperbolic polyhedron $\mathcal P$ has $\ver(\mathcal P)$ vertices, then 
\begin{equation} \label{eqn3}
\frac{\voct}{4}(\ver(\mathcal P) - 2) \leq \vol(\mathcal P) \leq \frac{\voct}{2}(\ver(\mathcal P) - 4).
\end{equation}
Both inequalities become equalities for the regular ideal hyperbolic octahedron---the unique ideal right-angled polyhedron with six vertices.
Moreover, the upper estimate is asymptotically sharp: there exists a sequence of ideal right-angled polyhedra $\{Q_i\}$ with $\ver(Q_i) \to \infty$ such that 
$
\lim_{i \to \infty} \frac{\vol(Q_i)}{\ver(Q_i)} = \frac{1}{2} \voct.
$
\end{theorem}

Note that, by Theorem~\ref{theorem2.1}, inequality~\eqref{eqn3} assumes that $\ver(\mathcal P) \geq 6$. Subsequently, the upper estimate in~(\ref{eqn3}) was improved in~\cite{EV20-2} in the case $\ver(\mathcal P) \geq 8$ and in~\cite{ABEV} in the case $\ver(\mathcal P) \geq 24$.
	
Consider the set $\mathcal R^n_{ideal} = \{ \mathcal P \in \mathcal R_{ideal} \, \mid \ver(\mathcal P) = n\}$, i.e., the set of ideal right-angled hyperbolic polyhedra having $n$ vertices. The volumes of polyhedra belonging to the sets $\mathcal R^n_{ideal}$, where $6 \leq n \leq 21$, were computed in~\cite{VeEg}.

In Table~\ref{table1}, for $6 \leq n \leq 21$, we give the following quantities: the number of elements $| \mathcal R^n_{ideal}| $ in the set $\mathcal R^n_{ideal}$; the number of distinct normalized volumes $|\Omega(\mathcal R^n_{ideal})|$; and the minimum $\min \Omega(\mathcal R^n_{ideal})$ and maximum $\max \Omega(\mathcal R^n_{ideal})$ values of normalized volumes.

\begin{table}[ht] \caption{Properties of normalized volumes in the ideal case.} \label{table1} 
\begin{center} 
\begin{tabular}{|r|r|r|r|r|} \hline
$n$ & $| \mathcal R^n_{ideal}| $ & $ |\Omega(\mathcal R^n_{ideal})|$ & $\min \Omega(\mathcal R^n_{ideal})$ & $\max \Omega(\mathcal R^n_{ideal})$  \\ \hline \hline 
6 & 1 & 1 & 0.610643  & 0.610643\\ \hline 
7 & 0 & 0 & - & -\\ \hline
8 & 1 & 1 & 0.752880 & 0.752880 \\ \hline
9 & 1 & 1 & 0.814191 & 0.814191 \\ \hline
10 & 2 & 2 & 0.813788  & 0.861241  \\ \hline
11 & 2 & 2 & 0.880628 & 0.922674   \\ \hline
12 & 9 & 7 &  0.845784 & 1.003841  \\ \hline
13 & 11 & 7 &  0.907826 & 1.026982  \\ \hline
14 & 37 & 17 & 0.864735 & 1.059477   \\ \hline
15 & 79 & 31 & 0.920885 & 1.088771   \\ \hline
16 & 249 & 79 & 0.876903 & 1.129321   \\ \hline
17 & 671 & 172 & 0.927656  & 1.148432  \\ \hline
18 & 2182 & 495 & 0.885188  & 1.180085  \\ \hline
19 & 6692 & 1359 & 0.931280 & 1.204969   \\ \hline
20 & 22131 & 4276 & 0.891085 & 1.229961   \\ \hline
21 & 72405 & 13031 & 0.933202 & 1.248958  \\ \hline
\end{tabular}
\end{center}
\end{table}

\section{Compact right-angled polyhedra} \label{prelim-compact}

 A \textit{compact right-angled polyhedron} is a right-angled hyperbolic polyhedron all of whose vertices are finite. By Theorem~\ref{theorem2.1}, all vertices of a compact right-angled polyhedron are 3-va\-lent. Note (see, for instance,~\cite[Th.~3.2]{In08}) that every face of a compact right-angled polyhedron is a right-angled polygon and, consequently, has at least $5$ sides.

For a polyhedron $\mathcal P \in \mathcal R_{comp}$, denote by $F$ the number of its faces, by $V$ the number of its vertices, and by $p_k$ the number of its $k$-gonal faces, $k \geq 5$. Then $F = \sum_{k\geq 5} p_k$. Euler's formula and the 3-valency of the vertices of the polyhedron imply that $2F = V + 4$, whence
\begin{equation} \label{eqn4}
p_5 = 12 + \sum_{k \geq 7} (k-6) p_k. 
\end{equation} 
In particular, the compact right-angled hyperbolic polyhedron with the smallest number of faces, and therefore with the smallest number of vertices, is the dodecahedron.
	
Since the conditions for realizing an abstract right-angled polyhedron as a compact right-angled polyhedron in $\mathbb{H}^3$ were first studied by Pogorelov~\cite{Pog}, these polyhedra are also called \textit{Pogorelov polyhedra}. Information on the life and scientific research of Aleksei Vasil'evich Pogorelov is presented in the article~\cite{BVI}, dedicated to the 100th anniversary of his birth. Note that fullerene polyhedra are Pogorelov polyhedra. The correlation of volumes of right-angled hyperbolic fullerenes with chemical properties of fullerenes is discussed in~\cite{EV20-1}.
		
In the following example, we describe an infinite family of compact right-angled polyhedra generalizing the dodecahedron, which will play an important role in our considerations.
	
\begin{example} \cite{Ves87} \label{ex:lobell} 
{\rm 
For $n \geq 5$, a \textit{L\"obell polyhedron} $L(n)$ is a $(2n+2)$-hedron with upper and lower $n$-gonal bases and with lateral surface formed by two layers of $n$ pentagons each, in which three edges meet at every vertex. The Schlegel diagrams of the polyhedra $L(5)$ and $L(6)$ are shown in Fig.~\ref{fig5}. In particular, $L(5)$ is the dodecahedron. Realizability of right-angled L\"obell polyhedra as hyperbolic polyhedra follows from Theorem~\ref{theorem2.1}. Right-angled hyperbolic L\"obell polyhedra will be denoted by $\mathcal L(n)$.
}	
\end{example}

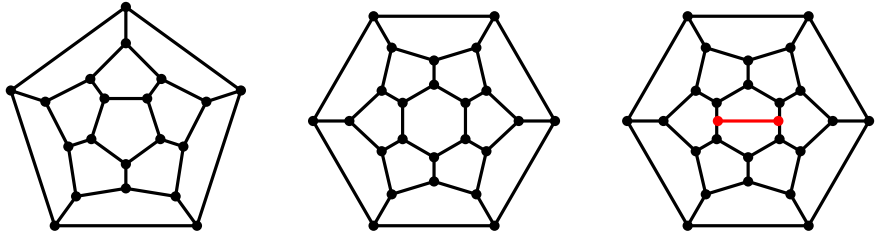
\begin{figure}[ht]	
\begin{center} 
\begin{tikzpicture}[scale=0.8, rotate = 180] 
\coordinate (O) at (0,0);
\foreach \angle [count=\i] in {234,306,...,594} {
\coordinate (A\i) at (\angle:0.6);}
\draw[very thick, black] (A1) -- (A2) -- (A3) -- (A4) -- (A5) -- cycle; 
\foreach \angle [count=\i] in {234,306,...,594} {
\coordinate (B\i) at (\angle:1.);}
\draw[very thick, black] (A1) -- (B1);
\draw[very thick, black] (A2) -- (B2);
\draw[very thick, black] (A3) -- (B3);
\draw[very thick, black] (A4) -- (B4);
\draw[very thick, black] (A5) -- (B5);
\foreach \angle [count=\i] in {54,126,...,342} {
\coordinate (C\i) at (\angle:1.4);}
\draw[very thick, black] (B1) -- (C3);
\draw[very thick, black] (B1) -- (C4);
\draw[very thick, black] (B2) -- (C4);
\draw[very thick, black] (B2) -- (C5);
\draw[very thick, black] (B3) -- (C5);
\draw[very thick, black] (B3) -- (C1);
\draw[very thick, black] (B4) -- (C1);
\draw[very thick, black] (B4) -- (C2);
\draw[very thick, black] (B5) -- (C2);
\draw[very thick, black] (B5) -- (C3);
\foreach \angle [count=\i] in {54,126,...,342} {
\coordinate (D\i) at (\angle:2.);}
\draw[very thick, black] (D1) -- (D2) -- (D3) -- (D4) -- (D5) -- cycle;
\draw[very thick, black] (C1) -- (D1);
\draw[very thick, black] (C2) -- (D2);
\draw[very thick, black] (C3) -- (D3);
\draw[very thick, black] (C4) -- (D4);
\draw[very thick, black] (C5) -- (D5);	
\foreach \v in {A1,A2,A3,A4,A5,B1,B2,B3,B4,B5,C1,C2,C3,C4,C5, D1, D2, D3, D4, D5}
{\fill[black] (\v) circle (2.4pt);}
\end{tikzpicture}
\qquad 
\begin{tikzpicture}[scale=0.8, rotate = 180] 
\coordinate (O) at (0,0);
\foreach \angle [count=\i] in {30,90,...,330} {
\coordinate (A\i) at (\angle:0.6);}
\draw[very thick, black] (A1) -- (A2) -- (A3) -- (A4) -- (A5) -- (A6) -- cycle; 
\foreach \angle [count=\i] in {30,90,...,330} {
\coordinate (B\i) at (\angle:1.);}
\draw[very thick, black] (A1) -- (B1);
\draw[very thick, black] (A2) -- (B2);
\draw[very thick, black] (A3) -- (B3);
\draw[very thick, black] (A4) -- (B4);
\draw[very thick, black] (A5) -- (B5);
\draw[very thick, black] (A6) -- (B6);
\foreach \angle [count=\i] in {0,60,...,360} {
\coordinate (C\i) at (\angle:1.4);}
\draw[very thick, black] (B1) -- (C1);
\draw[very thick, black] (B1) -- (C2);
\draw[very thick, black] (B2) -- (C2);
\draw[very thick, black] (B2) -- (C3);
\draw[very thick, black] (B3) -- (C3);
\draw[very thick, black] (B3) -- (C4);
\draw[very thick, black] (B4) -- (C4);
\draw[very thick, black] (B4) -- (C5);
\draw[very thick, black] (B5) -- (C5);
\draw[very thick, black] (B5) -- (C6);
\draw[very thick, black] (B6) -- (C6);
\draw[very thick, black] (B6) -- (C1);
\foreach \angle [count=\i] in {0,60,...,360} {
\coordinate (D\i) at (\angle:2.);}
\draw[very thick, black] (D1) -- (D2) -- (D3) -- (D4) -- (D5) -- (D6)-- cycle;
\draw[very thick, black] (C1) -- (D1);
\draw[very thick, black] (C2) -- (D2);
\draw[very thick, black] (C3) -- (D3);
\draw[very thick, black] (C4) -- (D4);
\draw[very thick, black] (C5) -- (D5); 
\draw[very thick, black] (C6) -- (D6);	
\foreach \v in {A1,A2,A3,A4,A5,A6,B1,B2,B3,B4,B5,B6,C1,C2,C3,C4,C5, C6, D1, D2, D3, D4, D5, D6}{
\fill[black] (\v) circle (2.4pt);}
\end{tikzpicture}
\qquad 
\begin{tikzpicture}[scale=0.8, rotate = 180] 
\coordinate (O) at (0,0);
\foreach \angle [count=\i] in {30,90,...,330} {
\coordinate (A\i) at (\angle:0.6);}
\draw[very thick, black] (A1) -- (A2) -- (A3) -- (A4) -- (A5) -- (A6) -- cycle; 
\foreach \angle [count=\i] in {30,90,...,330} {
\coordinate (B\i) at (\angle:1.);}
\draw[very thick, black] (A1) -- (B1);
\draw[very thick, black] (A2) -- (B2);
\draw[very thick, black] (A3) -- (B3);
\draw[very thick, black] (A4) -- (B4);
\draw[very thick, black] (A5) -- (B5);
\draw[very thick, black] (A6) -- (B6);
\foreach \angle [count=\i] in {0,60,...,360} {
\coordinate (C\i) at (\angle:1.4);}
\draw[very thick, black] (B1) -- (C1);
\draw[very thick, black] (B1) -- (C2);
\draw[very thick, black] (B2) -- (C2);
\draw[very thick, black] (B2) -- (C3);
\draw[very thick, black] (B3) -- (C3);
\draw[very thick, black] (B3) -- (C4);
\draw[very thick, black] (B4) -- (C4);
\draw[very thick, black] (B4) -- (C5);
\draw[very thick, black] (B5) -- (C5);
\draw[very thick, black] (B5) -- (C6);
\draw[very thick, black] (B6) -- (C6);
\draw[very thick, black] (B6) -- (C1);
\foreach \angle [count=\i] in {0,60,...,360} {
\coordinate (D\i) at (\angle:2.);}
\draw[very thick, black] (D1) -- (D2) -- (D3) -- (D4) -- (D5) -- (D6)-- cycle;
\draw[very thick, black] (C1) -- (D1);
\draw[very thick, black] (C2) -- (D2);
\draw[very thick, black] (C3) -- (D3);
\draw[very thick, black] (C4) -- (D4);
\draw[very thick, black] (C5) -- (D5); 
\draw[very thick, black] (C6) -- (D6);	
\foreach \v in {A1,A2,A3,A4,A5,A6,B1,B2,B3,B4,B5,B6,C1,C2,C3,C4,C5, C6, D1, D2, D3, D4, D5, D6}{
\fill[black] (\v) circle (2.4pt);}
\coordinate (M) at (-0.5,0);
\coordinate (N) at (0.5,0);
\draw[very thick, red] (M) -- (N);	
\fill[red] (M) circle (2.4pt);
\fill[red] (N) circle (2.4pt);
\end{tikzpicture}
\end{center}
\caption{The dodecahedron $L(5)$, the polyhedron $L(6)$, and the polyhedron $L(6)^+$.} 
\label{fig5} 	
\end{figure}

As shown in~\cite{Ves98}, the volume of the polyhedron $\mathcal L(n)$, $n \geq 5$, is expressed by the following formula:
\begin{equation} \label{eqn5}
\vol(\mathcal L(n)) = \frac{n}{2} \left[ 2 \Lambda(\theta_n) + \Lambda\left( \theta_n + \frac{\pi}{n} \right) + \Lambda \left( \theta_n - \frac{\pi}{n} \right) - \Lambda\left( 2 \theta_n - \frac{\pi}{2} \right) \right], 
\end{equation} 
where $\theta_n = \frac{\pi}{2} - \arccos \left( \frac{1}{2\cos(\pi/n)} \right)$. In particular, $\vol(\mathcal L(5)) = 4.306210$ and $\vol(\mathcal L(6)) = 6.023046$. Geometric and arithmetic properties of the polyhedra $\mathcal L(n)$ are discussed in~\cite{BoDo}. For geometric properties of nearly right-angled L\"obell polyhedra, see~\cite{BuMeVe}.

Following~\cite{In08}, we define two operations on the set of abstract polyhedra: composition and edge surgery. Let $P_1$ and $P_2$ be two abstract polyhedra. Let $F_1$ be a face in $P_1$ having the same number of sides as some face $F_2$ in $P_2$. Define a new abstract polyhedron $P = P_1 \# P_2$ as the result of the combinatorial gluing of the polyhedra $P_1$ and $P_2$ along the faces $F_1$ and $F_2$. The polyhedron $P$ will be called the \textit{connected sum}, or \textit{composition}, of the polyhedra $P_1$ and $P_2$. As shown in~\cite[Th.~4.4]{In08}, if the polyhedra $P_1$ and $P_2$ are realizable as compact right-angled hyperbolic polyhedra, then $P$ is also realizable as a compact right-angled hyperbolic polyhedron.
A \textit{decomposition} is the operation inverse to composition, i.e., cutting the polyhedron $P$ along some prismatic $k$-circuit, $k \geq 5$, into two polyhedra $P_1$ and $P_2$, each of which is realizable as a compact right-angled hyperbolic polyhedron. As shown in~\cite[Th.~6.4]{In08}, if a compact right-angled hyperbolic polyhedron $\mathcal P$ is a composition of compact right-angled hyperbolic polyhedra $\mathcal P_1$ and $\mathcal P_2$, then $\vol(\mathcal P) \geq \vol(\mathcal P_1) + \vol(\mathcal P_2)$.

Let $F_1$ and $F_2$ be two distinct faces in a compact right-angled hyperbolic polyhedron $\mathcal P$. We will say that $F_1$ and $F_2$ are \textit{joined by an edge} if they are not adjacent and there exists an edge $e$ in $\mathcal P$ joining a vertex $v_1$ on the boundary of $F_1$ to a vertex $v_2$ on the boundary of $F_2$. Since all vertices in $\mathcal P$ are 3-valent, every edge joins only one pair of faces. An edge $e$ is called \textit{good} if $e$ joins faces with at least $6$ sides, and \textit{very good} if, in addition, it is not part of a prismatic $5$-circuit. \textit{Edge surgery} along the edge $e$ is the deletion of $e$ with the corresponding change in the numbers of vertices, edges, and faces (Fig.~\ref{fig6}). As shown in~\cite[Th.~7.1, Th.~8.6]{In08}, if $\mathcal P_0$ is a compact right-angled hyperbolic polyhedron and $e \in \mathcal P_0$ is a very good edge, then the polyhedron $\mathcal P_1$ obtained by surgery along the edge $e$ is also compact right-angled hyperbolic and, moreover, $\vol(\mathcal P_0) > \vol(\mathcal P_1)$.
 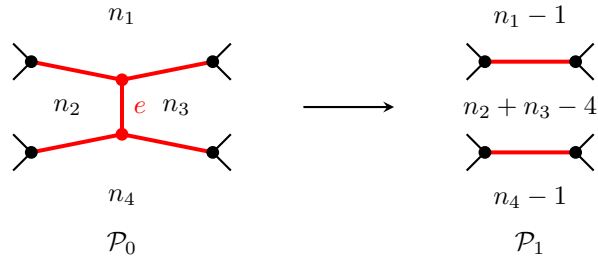
\begin{figure}[ht]
    \centering
    \setlength{\unitlength}{1.0mm}
    \begin{tikzpicture}[scale=1.2] 
\coordinate (1) at (5, 0);
\coordinate (2) at (6, 0);
\coordinate (3) at (5, 1);
\coordinate (4) at (6, 1);  
\draw[line width=1.5pt, red] (1) -- (2) ;
\draw[line width=1.5pt, red] (3) -- (4) ; 
 \draw[thick](5,1) -- (4.8,1.2);
\draw[thick](5,1) -- (4.8,0.8);
\draw[thick](6,1) -- (6.2,1.2);
\draw[thick](6,1) -- (6.2,0.8);
\draw[thick](6,0) -- (6.2,0.2);
\draw[thick](6,0) -- (6.2,-0.2);
 \draw[thick](5,0) -- (4.8,-0.2);
 \draw[thick](5,0) -- (4.8,0.2);
  \fill[black] (1) circle (2pt) ;
  \fill[black] (2) circle (2pt) ; 
  \fill[black] (3) circle (2pt) ;
  \fill[black] (4) circle (2pt) ;  
  \draw[-stealth,thick] (3,0.5) -- (4, 0.5);
\coordinate (6) at (0, 0) ;
  \coordinate (7) at (2, 0) ;
  \coordinate (8) at (0, 1) ;
  \coordinate (9) at (2, 1) ;
  \draw[line width=1.5pt, red] (0,1) -- (1,0.8) -- (2,1);
  \draw[line width=1.5pt, red] (0,0) -- (1,0.2) -- (2,0);
  \draw[line width=1.5pt, red] (1,0.2) -- (1,0.8) ; 
  \fill[black] (6) circle (2pt) ;
  \fill[black] (7) circle (2pt) ; 
  \fill[black] (8) circle (2pt) ;
  \fill[black] (9) circle (2pt) ;
  \fill[red] (1, 0.2) circle (2pt) ; 
   \fill[red] (1,0.8) circle (2pt) ; 
    \draw[thick](0,1) -- (-0.2,1.2);
    \draw[thick](0,1) -- (-0.2,0.8);
    \draw[thick](2,1) -- (2.2,1.2);
    \draw[thick](2,1) -- (2.2,0.8);
    \draw[thick](2,0) -- (2.2,0.2);
    \draw[thick](2,0) -- (2.2,-0.2);
    \draw[thick](0,0) -- (-0.2,-0.2);
    \draw[thick](0,0) -- (-0.2,0.2); 
    \draw(5.5,1.5) node {$n_1-1$};
     \draw(5.5,0.5) node {$n_2 + n_3 - 4$};
      \draw(5.5,-0.5) node {$n_4 - 1$};
         \draw(5.5,-1) node {$\mathcal P_1$};
     \draw(1,1.5) node {$n_1$};
     \draw(0.4,0.5) node {$n_2$};
      \draw(1.6,0.5) node {$n_3$};
     \draw(1,-0.5) node {$n_4$};
      \draw(1,-1.) node {$\mathcal P_0$};
     \draw[red](1.2,0.5) node {$e$};
    \end{tikzpicture}
    \caption{Edge surgery along the edge $e$.}
    \label{fig6}   
\end{figure}
The inverse operation will be called \textit{adding the edge $e$}, and the resulting polyhedron will be denoted by $\mathcal P^+ = \mathcal P \cup \{e\}$. In Fig.~\ref{fig5}, we show the Schlegel diagram of the polyhedron $L(6)^+$ obtained from $L(6)$ by adding an edge.

The important role of L\"obell polyhedra for describing the structure of the set of compact right-angled hyperbolic polyhedra is due to the following theorem.

\begin{theorem} \cite[Th.~9.1]{In08} \label{theorem3.1}
Let $\mathcal P_0$ be a compact right-angled hyperbolic polyhedron. Then there exists a sequence $\mathcal P_1, \mathcal P_2, \ldots, \mathcal P_k$ of pairwise disjoint unions of compact right-angled hyperbolic polyhedra such that, for $i=1, \ldots, k$, the union $\mathcal P_i$ is obtained from $\mathcal P_{i-1}$ by a decomposition or by edge surgery, while $\mathcal P_k$ consists of L\"obell polyhedra. Moreover,
$$
\vol(\mathcal P_0) \geq \vol(\mathcal P_1) \geq \vol(\mathcal P_2) \geq \dots \geq \vol(\mathcal P_k). 
$$
\end{theorem}

In Fig.~\ref{fig7}, we show an initial subgraph of the graph from~\cite{In22}, which illustrates four generations of polyhedra obtained from $\mathcal L(6)$ by edge-addition operations. The numbers assigned to the vertices correspond to the ordinal numbers of the polyhedra in the list of the first $825$ polyhedra from~\cite{In22}. Two vertices are joined by an edge if the polyhedra corresponding to these vertices are related by one edge-addition transformation.
\begin{figure}[h]
	\centering
\begin{tikzpicture}[main_node/.style={circle,draw,minimum size=0.8em, inner sep=1.2pt]}] 
\node[main_node] (2) at (0, 0) {};
\node[main_node] (3) at (0, 1) {};
\node[main_node] (5) at (-0.4, 2) {};
\node[main_node] (6) at (0.4, 2.2) {};
\node[main_node] (8) at (-1.5, 3) {};
\node[main_node] (9) at (0, 3.2) {};
\node[main_node] (10) at (1.5, 3.4) {};
\node[main_node] (13) at (-4.5, 4.) {};
\node[main_node] (14) at (-3.5, 4.1) {};
\node[main_node] (15) at (-2.5, 4.2) {};
\node[main_node] (16) at (-1.5, 4.3) {};
\node[main_node] (17) at (-0.5, 4.4) {};
\node[main_node] (19) at (0.5, 4.5){};
\node[main_node] (20) at (1.5, 4.6) {};
\node[main_node] (21) at (2.5, 4.7) {};
\node[main_node] (22) at (3.5, 4.8) {};
\node[main_node] (24) at (4.5, 4.9) {};
 \path[draw, thick]
(2) edge node {} (3)
(3) edge node {} (5)
(3) edge node {} (6)
(5) edge node {} (8)
(5) edge node {} (9)
(5) edge node {} (10)
(6) edge node {} (10)
(8) edge node {} (13)
(8) edge node {} (14)
(8) edge node {} (15)
(8) edge node {} (16)
(8) edge node {} (17)
(8) edge node {} (19)
(8) edge node {} (20)
(8) edge node {} (22)
(8) edge node {} (24)
(9) edge node {} (14)
(9) edge node {} (15)
(9) edge node {} (19)
(9) edge node {} (20)
(9) edge node {} (22)
(9) edge node {} (24)
(10) edge node {} (19)
(10) edge node {} (21)
(10) edge node {} (22)
(10) edge node {} (24)
;  
\draw(0,0) node {\tiny{2}};
\draw(0,1) node {\tiny{3}};
\draw(-0.4,2) node {\tiny{5}};
\draw(0.4,2.2) node {\tiny{6}};
\draw(-1.5,3) node {\tiny{8}};
\draw(0,3.2) node {\tiny{9}};
 \draw(1.5,3.4) node {\tiny{10}};
\draw(-4.5,4) node {\tiny{13}};
\draw(-3.5,4.1) node {\tiny{14}};
\draw(-2.5,4.2) node {\tiny{15}};
\draw(-1.5,4.3) node {\tiny{16}};
\draw(-0.5,4.4) node {\tiny{17}};
\draw(0.5,4.5) node {\tiny{19}};
\draw(1.5,4.6) node {\tiny{20}};
\draw(2.5,4.7) node {\tiny{21}};
\draw(3.5,4.8) node {\tiny{22}};
\draw(4.5,4.9) node {\tiny{24}};
\draw(0.8,0) node {\tiny $\mathcal L(6)$};
\draw(0.8,1) node {\tiny $\mathcal L(6)^+$};
\end{tikzpicture}
\caption{Initial part of the growth graph of polyhedra from $\mathcal L(6)$.}
\label{fig7}
\end{figure}
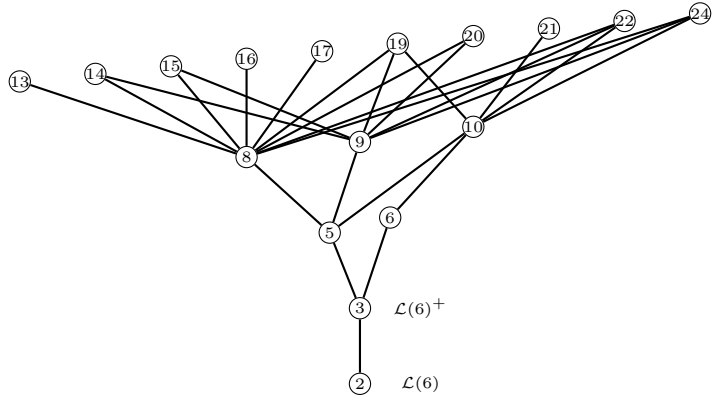

\begin{remark} \label{remark3.1}
{\rm 
Under an edge-addition operation, the normalized volume can both increase and decrease. An example of an increase in the normalized volume is given by the pair $\mathcal L(6)$ and $\mathcal L(6)^+$; see Fig.~\ref{fig5}. The inequality $\omega(\mathcal L(6)^+) > \omega(\mathcal L(6))$ holds, since $\omega(\mathcal L(6)) = \frac{6.023046}{24} = 0.250960$, while $\omega(\mathcal L(6)^+) = \frac{6.967011}{26} = 0.267961$. An example of a decrease in the normalized volume is given by a pair of polyhedra $\mathcal Q$ and $\mathcal Q^+$; see Fig.~\ref{fig107}, where $\mathcal Q^+$ is obtained by adding an edge joining the edges $e_1$ and $e_2$ of the polyhedron $\mathcal Q$, which has $42$ vertices. The inequality $\omega(\mathcal Q) > \omega(\mathcal Q^+)$ holds, since $\omega(\mathcal Q) = \frac{12.713430}{42} = 0.302700$, while $\omega(\mathcal Q^+) = \frac{12.996118}{44} = 0.295366$.
} 
\end{remark}

\begin{figure}[h]
	\centering
\begin{tikzpicture}[main_node/.style={circle,draw,minimum size=0.1 em, inner sep=1.2pt]}, scale=0.6 ] 
\node[main_node, fill=black ] (0) at (3.11, 0.88) {};
\node[main_node, fill=black ] (1) at (1.47, -1.51) {};
\node[main_node, fill=black ] (2) at (2.14, 3.49) {};
\node[main_node, fill=black ] (3) at (2.36, 0.79) {};
\node[main_node, fill=black ] (4) at (0.82, -0.90) {};
\node[main_node, fill=black ] (5) at (-0.33, -2.57) {};
\node[main_node, fill=black ] (6) at (-0.29, 4.88) {};
\node[main_node, fill=black ] (7) at (1.53, 3.02) {};
\node[main_node, fill=black ] (8) at (1.77, 1.78) {};
\node[main_node, fill=black ] (9) at (1.64, 0.05) {};
\node[main_node, fill=black ] (10) at (-0.08, -0.00) {};
\node[main_node, fill=black ] (11) at (-0.66, -1.71) {};
\node[main_node, fill=black ] (12) at (-3.07, -2.63) {};
\node[main_node, fill=black ] (13) at (-3.07, 4.40) {};
\node[main_node, fill=black ] (14) at (-0.60, 3.99) {};
\node[main_node, fill=black ] (15) at (0.57, 2.63) {};
\node[main_node, fill=black ] (16) at (1.18, 1.48) {};
\node[main_node, fill=black ] (17) at (0.96, 0.47) {};
\node[main_node, fill=black] (18) at (0.12, 0.62) {};
\node[main_node, fill=black] (19) at (-0.98, -0.32) {};
\node[main_node, fill=black] (20) at (-1.74, -1.67) {};
\node[main_node, fill=black] (21) at (-2.65, -1.95) {};
\node[main_node, fill=black] (22) at (-4.88, -0.50) {};
\node[main_node, fill=black] (23) at (-4.88, 2.27) {};
\node[main_node, fill=black] (24) at (-2.63, 3.53) {};
\node[main_node, fill=black] (25) at (-1.57, 3.29) {};
\node[main_node, fill=black] (26) at (-0.20, 2.95) {};
\node[main_node, fill=black] (27) at (0.62, 1.93) {};
\node[main_node, fill=black] (28) at (0.07, 1.61) {};
\node[main_node, fill=black] (29) at (-1.83, 0.14) {};
\node[main_node, fill=black] (30) at (-1.93, -0.82) {};
\node[main_node, fill=black] (31) at (-3.20, -1.00) {};
\node[main_node, fill=black] (32) at (-4.13, -0.13) {};
\node[main_node, fill=black] (33) at (-3.93, 2.00) {};
\node[main_node, fill=black] (34) at (-2.92, 2.42) {};
\node[main_node, fill=black] (35) at (-1.44, 2.28) {};
\node[main_node, fill=black] (36) at (-0.70, 2.20) {};
\node[main_node, fill=black] (37) at (-2.10, 0.77) {};
\node[main_node, fill=black] (38) at (-2.70, -0.42) {};
\node[main_node, fill=black] (39) at (-3.81, 1.11) {};
\node[main_node, fill=black] (40) at (-2.19, 1.73) {};
\node[main_node, fill=black] (41) at (-3.04, 0.47) {};
 \path[draw, thick]
(0) edge node {} (1) 
(0) edge node {} (2) 
(0) edge node {} (3) 
(1) edge node {} (4) 
(1) edge node {} (5) 
(2) edge node {} (6) 
(2) edge node {} (7) 
(3) edge node {} (8) 
(3) edge node {} (9) 
(4) edge node {} (9) 
(4) edge node {} (10) 
(5) edge node {} (11) 
(5) edge node {} (12) 
(6) edge node {} (13) 
(6) edge node {} (14) 
(7) edge node {} (8) 
(7) edge node {} (15) 
(8) edge node {} (16) 
(9) edge node {} (17) 
(10) edge node {} (18) 
(10) edge node {} (19) 
(11) edge node {} (19) 
(11) edge node {} (20) 
(12) edge node {} (21) 
(12) edge node {} (22) 
(13) edge node {} (23) 
(13) edge node {} (24) 
(14) edge node {} (25) 
(14) edge node {} (26) 
(15) edge node {} (26) 
(15) edge node {} (27) 
(16) edge node {} (17) 
(16) edge node {} (27) 
(17) edge node {} (18) 
(18) edge node {} (28) 
(19) edge node {} (29) 
(20) edge node {} (21) 
(20) edge node {} (30) 
(21) edge node {} (31) 
(22) edge node {} (23) 
(22) edge node {} (32) 
(23) edge node {} (33) 
(24) edge node {} (25) 
(24) edge node {} (34) 
(25) edge node {} (35) 
(26) edge node {} (36) 
(27) edge node {} (28) 
(28) edge node {} (36) 
(29) edge node {} (30) 
(29) edge node {} (37) 
(30) edge node {} (38) 
(31) edge node {} (32) 
(31) edge node {} (38) 
(32) edge node {} (39) 
(33) edge node {} (34) 
(33) edge node {} (39) 
(34) edge node {} (40) 
(35) edge node {} (36) 
(35) edge node {} (40) 
(37) edge node {} (40) 
(37) edge node {} (41) 
(38) edge node {} (41) 
(39) edge node {} (41);
 \draw(-0.6,0.3) node {$e_1$};
  \draw(0.6,-2.5) node {$e_2$};
\end{tikzpicture}
\qquad 
\qquad 
\begin{tikzpicture}[main_node/.style={circle,draw,minimum size=0.1 em, inner sep=1.2pt]}, scale=0.6 ] 
\node[main_node, fill=black ] (0) at (3.11, 0.88) {};
\node[main_node, fill=black ] (1) at (1.47, -1.51) {};
\node[main_node, fill=black ] (2) at (2.14, 3.49) {};
\node[main_node, fill=black ] (3) at (2.36, 0.79) {};
\node[main_node, fill=black ] (4) at (0.82, -0.90) {};
\node[main_node, fill=black ] (5) at (-0.33, -2.57) {};
\node[main_node, fill=black ] (6) at (-0.29, 4.88) {};
\node[main_node, fill=black ] (7) at (1.53, 3.02) {};
\node[main_node, fill=black ] (8) at (1.77, 1.78) {};
\node[main_node, fill=black ] (9) at (1.64, 0.05) {};
\node[main_node, fill=black ] (10) at (-0.08, -0.00) {};
\node[main_node, fill=black ] (11) at (-0.66, -1.71) {};
\node[main_node, fill=black ] (12) at (-3.07, -2.63) {};
\node[main_node, fill=black ] (13) at (-3.07, 4.40) {};
\node[main_node, fill=black ] (14) at (-0.60, 3.99) {};
\node[main_node, fill=black ] (15) at (0.57, 2.63) {};
\node[main_node, fill=black ] (16) at (1.18, 1.48) {};
\node[main_node, fill=black ] (17) at (0.96, 0.47) {};
\node[main_node, fill=black] (18) at (0.12, 0.62) {};
\node[main_node, fill=black] (19) at (-0.98, -0.32) {};
\node[main_node, fill=black] (20) at (-1.74, -1.67) {};
\node[main_node, fill=black] (21) at (-2.65, -1.95) {};
\node[main_node, fill=black] (22) at (-4.88, -0.50) {};
\node[main_node, fill=black] (23) at (-4.88, 2.27) {};
\node[main_node, fill=black] (24) at (-2.63, 3.53) {};
\node[main_node, fill=black] (25) at (-1.57, 3.29) {};
\node[main_node, fill=black] (26) at (-0.20, 2.95) {};
\node[main_node, fill=black] (27) at (0.62, 1.93) {};
\node[main_node, fill=black] (28) at (0.07, 1.61) {};
\node[main_node, fill=black] (29) at (-1.83, 0.14) {};
\node[main_node, fill=black] (30) at (-1.93, -0.82) {};
\node[main_node, fill=black] (31) at (-3.20, -1.00) {};
\node[main_node, fill=black] (32) at (-4.13, -0.13) {};
\node[main_node, fill=black] (33) at (-3.93, 2.00) {};
\node[main_node, fill=black] (34) at (-2.92, 2.42) {};
\node[main_node, fill=black] (35) at (-1.44, 2.28) {};
\node[main_node, fill=black] (36) at (-0.70, 2.20) {};
\node[main_node, fill=black] (37) at (-2.10, 0.77) {};
\node[main_node, fill=black] (38) at (-2.70, -0.42) {};
\node[main_node, fill=black] (39) at (-3.81, 1.11) {};
\node[main_node, fill=black] (40) at (-2.19, 1.73) {};
\node[main_node, fill=black] (41) at (-3.04, 0.47) {};
\node[main_node, fill=red] (42) at (-0.40, -0.35) {};
\node[main_node, fill=red] (43) at (0.60, -2.00) {};
 \path[draw, thick]
(0) edge node {} (1) 
(0) edge node {} (2) 
(0) edge node {} (3) 
(1) edge node {} (4) 
(2) edge node {} (6) 
(2) edge node {} (7) 
(3) edge node {} (8) 
(3) edge node {} (9) 
(4) edge node {} (9) 
(4) edge node {} (10) 
(5) edge node {} (11) 
(5) edge node {} (12) 
(6) edge node {} (13) 
(6) edge node {} (14) 
(7) edge node {} (8) 
(7) edge node {} (15) 
(8) edge node {} (16) 
(9) edge node {} (17) 
(10) edge node {} (18) 
(11) edge node {} (19) 
(11) edge node {} (20) 
(12) edge node {} (21) 
(12) edge node {} (22) 
(13) edge node {} (23) 
(13) edge node {} (24) 
(14) edge node {} (25) 
(14) edge node {} (26) 
(15) edge node {} (26) 
(15) edge node {} (27) 
(16) edge node {} (17) 
(16) edge node {} (27) 
(17) edge node {} (18) 
(18) edge node {} (28) 
(19) edge node {} (29) 
(20) edge node {} (21) 
(20) edge node {} (30) 
(21) edge node {} (31) 
(22) edge node {} (23) 
(22) edge node {} (32) 
(23) edge node {} (33) 
(24) edge node {} (25) 
(24) edge node {} (34) 
(25) edge node {} (35) 
(26) edge node {} (36) 
(27) edge node {} (28) 
(28) edge node {} (36) 
(29) edge node {} (30) 
(29) edge node {} (37) 
(30) edge node {} (38) 
(31) edge node {} (32) 
(31) edge node {} (38) 
(32) edge node {} (39) 
(33) edge node {} (34) 
(33) edge node {} (39) 
(34) edge node {} (40) 
(35) edge node {} (36) 
(35) edge node {} (40) 
(37) edge node {} (40) 
(37) edge node {} (41) 
(38) edge node {} (41) 
(39) edge node {} (41);
 \path[draw, very thick, red]
(42) edge node {} (19)
(42) edge node {} (10);
 \path[draw, very thick, red]
(43) edge node {} (1)
(43) edge node {} (5)
(43) edge node {} (42);
\end{tikzpicture}
	\caption{The polyhedra $\mathcal Q$ and $\mathcal Q^+$.}
	\label{fig107}
\end{figure}
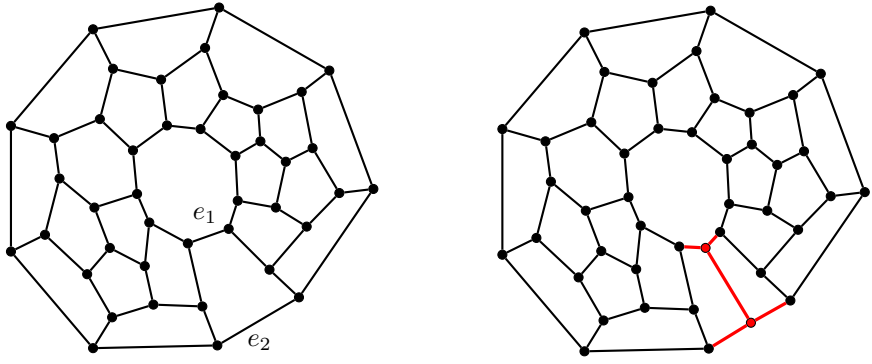	
	
Below we will use L\"obell polyhedra $\mathcal L(n)$ to construct new compact right-angled polyhedra.
		
\begin{example} \label{ex:tower}
{\rm For an integer $k \geq 1$, denote by $\mathcal L_k(n)$ the polyhedron with $2n(k+1)$ vertices constructed from $k$ copies of the polyhedron $\mathcal L(n)$ by successively gluing them along $n$-gonal faces. In particular, $\mathcal L_1(n) = \mathcal L(n)$. As an example, in Fig.~\ref{fig9} we show a net of the lateral surface of the polyhedron $\mathcal L_3(6)$. The polyhedra $\mathcal L_k(n)$ will be called \textit{towers of L\"obell polyhedra}.
}
 \end{example}

\begin{figure}[ht]
    \centering
    \setlength{\unitlength}{1.0mm}
    \begin{tikzpicture}[scale=1.] 
\coordinate (A0) at (0, 0) ;
\coordinate (A1) at (1, 0) ;
\coordinate (A2) at (2, 0) ;
\coordinate (A3) at (3, 0) ;
\coordinate (A4) at (4, 0) ;
 \coordinate (A5) at (5, 0) ;
  \coordinate (A6) at (6, 0) ;
\foreach \v in {A0,A1,A2,A3,A4,A5,A6}{
\fill[black] (\v) circle (2pt);}
\coordinate (B0) at (0, 0.5) ;
\coordinate (B1) at (1, 0.5) ;
\coordinate (B2) at (2, 0.5) ;
\coordinate (B3) at (3, 0.5) ;
\coordinate (B4) at (4, 0.5) ;
 \coordinate (B5) at (5, 0.5) ;
  \coordinate (B6) at (6, 0.5) ;
\foreach \v in {B0,B1,B2,B3,B4,B5,B6}{
\fill[black] (\v) circle (2pt);}
\coordinate (C0) at (0.5, 1) ;
\coordinate (C1) at (1.5, 1) ;
\coordinate (C2) at (2.5, 1) ;
\coordinate (C3) at (3.5, 1) ;
\coordinate (C4) at (4.5, 1) ;
 \coordinate (C5) at (5.5, 1) ;
  \coordinate (C6) at (6.5, 1) ;
\foreach \v in {C0,C1,C2,C3,C4,C5,C6}{
\fill[black] (\v) circle (2pt);}
\coordinate (D0) at (0.5, 2) ;
\coordinate (D1) at (1.5, 2) ;
\coordinate (D2) at (2.5, 2) ;
\coordinate (D3) at (3.5, 2) ;
\coordinate (D4) at (4.5, 2) ;
 \coordinate (D5) at (5.5, 2) ;
  \coordinate (D6) at (6.5, 2) ;
\foreach \v in {D0,D1,D2,D3,D4,D5,D6}{
\fill[black] (\v) circle (2pt);}
\coordinate (E0) at (0, 2.5) ;
\coordinate (E1) at (1, 2.5) ;
\coordinate (E2) at (2, 2.5) ;
\coordinate (E3) at (3, 2.5) ;
\coordinate (E4) at (4, 2.5) ;
 \coordinate (E5) at (5, 2.5) ;
  \coordinate (E6) at (6, 2.5) ;
\foreach \v in {E0,E1,E2,E3,E4,E5,E6}{
\fill[black] (\v) circle (2pt);}
\coordinate (G0) at (0, 3.5) ;
\coordinate (G1) at (1, 3.5) ;
\coordinate (G2) at (2, 3.5) ;
\coordinate (G3) at (3, 3.5) ;
\coordinate (G4) at (4, 3.5) ;
 \coordinate (G5) at (5, 3.5) ;
  \coordinate (G6) at (6, 3.5) ;
\foreach \v in {G0,G1,G2,G3,G4,G5,G6}{
\fill[black] (\v) circle (2pt);}
\coordinate (H0) at (0.5, 4) ;
\coordinate (H1) at (1.5, 4) ;
\coordinate (H2) at (2.5, 4) ;
\coordinate (H3) at (3.5, 4) ;
\coordinate (H4) at (4.5, 4) ;
 \coordinate (H5) at (5.5, 4) ;
  \coordinate (H6) at (6.5, 4) ;
\foreach \v in {H0,H1,H2,H3,H4,H5,H6}{
\fill[black] (\v) circle (2pt);}
\coordinate (K0) at (0.5, 4.5) ;
\coordinate (K1) at (1.5, 4.5) ;
\coordinate (K2) at (2.5, 4.5) ;
\coordinate (K3) at (3.5, 4.5) ;
\coordinate (K4) at (4.5, 4.5) ;
 \coordinate (K5) at (5.5, 4.5) ;
  \coordinate (K6) at (6.5, 4.5) ;
\foreach \v in {K0,K1,K2,K3,K4,K5,K6}{
\fill[black] (\v) circle (2pt);}
\draw[very thick] (A0) -- (B0) -- (C0) -- (D0) -- (E0) -- (G0) -- (H0) --(K0);
\draw[very thick] (A1) -- (B1) -- (C1) -- (D1) -- (E1) -- (G1) -- (H1) --(K1);
\draw[very thick] (A2) -- (B2) -- (C2) -- (D2) -- (E2) -- (G2) -- (H2) --(K2);
\draw[very thick] (A3) -- (B3) -- (C3) -- (D3) -- (E3) -- (G3) -- (H3) --(K3);
\draw[very thick] (A4) -- (B4) -- (C4) -- (D4) -- (E4) -- (G4) -- (H4) --(K4);
\draw[very thick] (A5) -- (B5) -- (C5) -- (D5) -- (E5) -- (G5) -- (H5) --(K5);
\draw[very thick] (A6) -- (B6) -- (C6) -- (D6) -- (E6) -- (G6) -- (H6) --(K6);
\draw[very thick] (A0) -- (A6); 
\draw[very thick] (K0) -- (K6); 
\draw[very thick] (C0) -- (B1); 
\draw[very thick] (C1) -- (B2); 
\draw[very thick] (C2) -- (B3); 
\draw[very thick] (C3) -- (B4); 
\draw[very thick] (C4) -- (B5);  
\draw[very thick] (C5) -- (B6);  
\draw[very thick] (D0) -- (E1); 
\draw[very thick] (D1) -- (E2); 
\draw[very thick] (D2) -- (E3); 
\draw[very thick] (D3) -- (E4); 
\draw[very thick] (D4) -- (E5);  
\draw[very thick] (D5) -- (E6); 
\draw[very thick] (H0) -- (G1); 
\draw[very thick] (H1) -- (G2); 
\draw[very thick] (H2) -- (G3); 
\draw[very thick] (H3) -- (G4); 
\draw[very thick] (H4) -- (G5);  
\draw[very thick] (H5) -- (G6); 
\draw[dashed, red] (-1,1.5) -- (7,1.5); 
\draw[dashed, red] (-1,3) -- (7,3); 
   \draw(-1,0.75) node {$\mathcal L(6)$};
      \draw(-1,2.25) node {$\mathcal L(6)$};
         \draw(-1,3.75) node {$\mathcal L(6)$};
    \end{tikzpicture}
    \caption{Lateral surface of the polyhedron $\mathcal L_3(6)$.}
    \label{fig9}   
\end{figure}
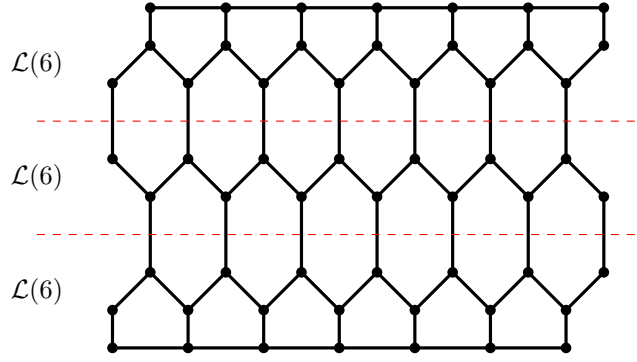
	
\begin{remark} \label{note}
{\rm The polyhedron $\mathcal L(n)$ has $2n$ mutually isometric pentagonal faces adjacent to the two $n$-gonal faces. By construction, the tower $\mathcal L_k(n)$ has $2n$ mutually isometric pentagonal faces of the same kind adjacent to the upper or lower $n$-gonal face of the tower.
}
\end{remark}
	
The following property was established in~\cite{VesRep}.

\begin{theorem} \cite[Theorem~3]{VesRep} \label{theorem300}
For every integer $k \geq 1$, the following property holds for towers of compact right-angled hyperbolic L\"obell polyhedra $\mathcal L_k(n)$:
\begin{equation} \label{eqn6}
\lim_{n\to\infty} \frac{\vol(\mathcal L_k(n))}{\ver(\mathcal L_k(n))} = \frac{k}{k+1} \cdot \frac{5\vtet}{8},
\end{equation}
where $\vtet = 3\Lambda(\frac{\pi}{3})$ is the volume of the regular ideal hyperbolic tetrahedron. 
\end{theorem}

In particular, for $k=1$ we obtain $\lim\limits_{n\to\infty} \omega(\mathcal L(n)) = \frac{5}{16}\vtet$, while $\lim\limits_{k,n\to\infty} \omega(\mathcal L_k(n)) = \frac{5}{8} \vtet$.
	
Estimates for volumes of compact right-angled polyhedra were obtained by Atkinson~\cite{At09}.

\begin{theorem} \cite[Theorem~2.3]{At09}.  \label{theorem3.2}
If a compact right-angled hyperbolic polyhedron $\mathcal P$ has $\ver(\mathcal P)$ vertices, then
\begin{equation} \label{eqn7}
\frac{\voct}{32}(\ver(\mathcal P)-8) \leq \vol(\mathcal P) < \frac{5\vtet}{8}(\ver(\mathcal P)-10),  
\end{equation}
where $\vtet = 3 \Lambda(\frac{\pi}{3})$ is the volume of the regular ideal hyperbolic tetrahedron. Moreover, there exists a sequence of polyhedra $\{R_i\} \subset \mathcal{R}_{comp}$ with $\ver(R_i)$ vertices such that $\lim_{i \to \infty} \frac{\vol(R_i)}{\ver(R_i)} = \frac{5}{8} \vtet$.
\end{theorem}
		
As noted above, the compact right-angled hyperbolic polyhedron with the smallest number of vertices is the dodecahedron. Thus, in formula~(\ref{eqn7}) it is assumed that $\ver(\mathcal P) \geq 20$. Subsequently, the upper estimate in~(\ref{eqn7}) was improved in~\cite{EV20-2} in the case $\ver(\mathcal P) \geq 24$ and in~\cite{ABEV} in the case $\ver(\mathcal P) \geq 81$.

Consider the set $\mathcal R^n_{comp} = \{ \mathcal P \in \mathcal R_{comp} \, \mid \ver(\mathcal P) = n\}$, i.e., the set of compact right-angled polyhedra having exactly $n$ vertices. Since the one-dimensional skeletons of polyhedra from $\mathcal R^n_{comp}$ are 3-valent graphs, the number of vertices $\ver(\mathcal P)$ is even. For even $20 \leq n \leq 46$, the number of elements $| \mathcal R^n_{comp}| $ in the set $\mathcal R^n_{comp}$, the number of distinct normalized volumes $|\Omega(\mathcal R^n_{comp})|$, and the minimum $\min \Omega(\mathcal R^n_{comp})$ and maximum $\max \Omega(\mathcal R^n_{comp})$ values of normalized volumes are given in Table~\ref{table2}.

\begin{table}[ht] \caption{Density of normalized volumes in the compact case.} \label{table2} 
\begin{tabular}{|r|r|r|r|r|} \hline
$n$ & $| \mathcal R^n_{comp}| $ & $ |\Omega(\mathcal R^n_{comp})|$ & $\min \Omega(\mathcal R^n_{comp})$ & $\max \Omega(\mathcal R^n_{comp})$  \\ \hline \hline 			
20 & 1 & 1 & 0.215310  & 0.215310 \\ \hline 
22 & 0 & 0 & - & -  \\ \hline 
24 & 1 & 1 & 0.250960  & 0.250960 \\ \hline 
26 & 1 & 1 & 0.267961 & 0.267961 \\ \hline
28 & 3 & 3 & 0.270116  & 0.285722 \\ \hline
30 & 4 & 4 & 0.287080  & 0.298220 \\ \hline
32 & 12 & 12 & 0.281845  & 0.311786 \\ \hline
34 & 23 & 23 & 0.286201 &0.323119  \\ \hline
36 & 71 & 71 & 0.289334 & 0.335671 \\ \hline
38 & 187 & 187 &  0.291020  & 0.345760 \\ \hline
40 & 627 & 627 & 0.292711 & 0.355566 \\ \hline
42 & 1970 & 1952 & 0.294115  & 0.364289  \\ \hline
44 & 6833 & 6771 & 0.295366 & 0.372678 \\ \hline
46 & 23384 & 23082 & 0.296473 & 0.380143 \\ \hline
\end{tabular}
\end{table}

\section{Proof of Theorem~\ref{theorem1}} \label{main-ideal}

\begin{proof}[Proof of Theorem~\ref{theorem1}]
	We show that $\Omega(\mathcal R_{ideal}) \subset \left[ \frac{1}{6} \voct, \frac{1}{2} \voct \right)$, where the lower bound is attained on the octahedron and the upper bound is sharp.
	Inequality~\eqref{eqn3} implies
	$$
	\frac{\voct}{4} \frac{\ver(\mathcal P)-2}{\ver(\mathcal P)} \leq \omega(\mathcal P) \leq \frac{\voct}{2} \frac{\ver(\mathcal P)-4}{\ver(\mathcal P)}, 
	$$
	where $\omega(\mathcal P) = \vol(\mathcal P)/\ver(\mathcal P)$ and $\ver(\mathcal P) \geq 6$.
	Consider the lower estimate. Since, for $N \geq 6$, the function $f(N) = \frac{N-2}{N}$ increases monotonically, we have $\min_{N\geq 6} f(N) = f(6) = \frac{2}{3}$, whence $\omega(\mathcal P) \geq \frac{1}{6} \voct$, and equality is attained on the octahedron.
	Consider the upper estimate. For $N \geq 6$, the function $g(N) = \frac{N-4}{N}$ increases monotonically and tends to $1$ as $N \to \infty$; hence, if $\ver(\mathcal P) \geq 6$, then the inequality $\omega(\mathcal P) < \frac{1}{2} \voct$ holds.
	At the same time, since, by Theorem~\ref{theorem2.3}, the upper estimate in inequality~(\ref{eqn3}) is asymptotically sharp, there exists a sequence of polyhedra whose normalized volumes tend to $\frac{1}{2}\voct$.
	Therefore, the upper bound of the spectrum is equal to $\frac{1}{2}\voct$, but it is not attained.
	
	We show that the spectrum is discrete on the interval $\left[ \frac{1}{6} \voct, \frac{1}{4} \voct \right)$.
	Indeed, given $\frac{1}{6} \voct < C < \frac{1}{4} \voct$, consider the interval $\left[ \frac{1}{6} \voct, C \right]$.
	Assume that $\omega(\mathcal P) \in \left[ \frac{1}{6} \voct, C \right]$.
	Then the lower estimate in inequality~(\ref{eqn3}) implies
	$$
	\frac{\voct}{4} \left(1 - \frac{2}{\ver(\mathcal P)}\right) \leq \omega(\mathcal P) \leq C.
	$$
	It follows that 
	$$
	1 - \frac{2}{\ver(\mathcal P)} \leq \frac{4C}{\voct} \quad \text{and hence} \quad \ver(\mathcal P) \leq \frac{2 \voct}{\voct-4C}.
	$$
	By the choice of $C$, the expression on the right-hand side of the inequality is positive.
	Thus, the number of vertices $\ver(\mathcal P)$ is bounded above by a quantity depending only on $C$ and the constant $\voct$.
	
	According to Theorem~\ref{theorem2.1}, an ideal right-angled hyperbolic polyhedron is determined uniquely up to isometry by its 1-skeleton, which is a 3-connected planar 4-valent graph.
	Since the number of such graphs with a given number of vertices is finite, the set of possible volume values, and hence of possible normalized volume values, is also finite.
	Thus, every interval $\left[ \frac{1}{6} \voct, C \right]$, where $C < \frac{1}{4} \voct$, contains only finitely many values of the spectrum.
	Consequently, the spectrum is discrete on the interval $\left[ \frac{1}{6} \voct, \frac{1}{4} \voct \right)$.

	Now we establish that the spectrum of normalized volumes is dense on the interval $\left[ \frac{1}{4} \voct, \frac{1}{2} \voct \right]$.
	For convenience of exposition, introduce auxiliary quantities. For a polyhedron $\mathcal P \in \mathcal R_{ideal}$, define the \textit{modified number of vertices} $\tN(\mathcal P)$ by the formula  
	$\tN(\mathcal P) = \ver(\mathcal P) - 3$, and also define the \emph{modified normalized volume} $\tomega(\mathcal P)$ by setting:
	$$
	\tomega(\mathcal P) = \frac{\vol(\mathcal P)}{\tN(\mathcal P)}.
	$$		
	By equality~(\ref{eqn1}), every polyhedron from $\mathcal{R}_{ideal}$ has at least eight triangular faces.
	It is well known (see, for instance,~\cite[Prop. 2.5]{FG2011}) that in Lobachevsky geometry every two ideal triangles are isometric to each other.
		This property makes it possible to correctly define the operation of a \textit{connected sum}, i.e., gluing two ideal right-angled polyhedra along their triangular faces.
	Let $\mathcal P_1, \mathcal P_2 \in \mathcal{R}_{ideal}$. Choose triangular faces $T_1 \subset \mathcal P_1$ and $T_2 \subset \mathcal P_2$.
	Since $T_1$ and $T_2$ are isometric, we can glue $\mathcal P_1$ and $\mathcal P_2$ along these faces.
	Denote the resulting polyhedron by $\mathcal P_1 \# \mathcal P_2$. 
	
	\begin{lemma} \label{lemma4.1}
		Under the operation of connected sum of ideal right-angled polyhedra along triangular faces, volume behaves additively: 
		$$
		\vol(\mathcal P_1 \# \mathcal P_2) = \vol(\mathcal P_1) + \vol(\mathcal P_2),
		$$
		and the number of vertices changes as follows:
		$$
		\ver(\mathcal P_1 \# \mathcal P_2) = \ver(\mathcal P_1) + \ver(\mathcal P_2) - 3.
		$$
		Moreover, the polyhedron $\mathcal P_1 \# \mathcal P_2$ is ideal right-angled.
	\end{lemma}
	\begin{proof}
			When right-angled polyhedra are glued, the faces adjacent to $T_1$ and $T_2$ pairwise form flat dihedral angles of size $\pi$ and merge into complete geodesic hyperfaces. The right angles with the remaining faces are preserved; hence $\mathcal P_1 \# \mathcal P_2$ is again a right-angled polyhedron.
			Since the two-dimensional gluing face has zero three-dimensional volume, the equality $\vol(\mathcal P_1 \# \mathcal P_2) = \vol(\mathcal P_1) + \vol(\mathcal P_2)$ holds. When the gluing faces are identified, the three ideal vertices of the triangle $T_1$ are identified pairwise with the three ideal vertices of the triangle $T_2$. Since these vertices lie on the absolute, after gluing they remain ideal vertices. Thus, the $6$ original vertices corresponding to $T_1$ and $T_2$ merge into $3$ vertices, which gives the formula $\ver(\mathcal P_1 \# \mathcal P_2) = \ver(\mathcal P_1) + \ver(\mathcal P_2) - 3$. The lemma is proved.
	\end{proof}
	
			Observe that under the operation of connected sum the modified number of vertices behaves additively:
	\begin{eqnarray*}
		\tN(\mathcal P_1 \# \mathcal P_2) & = & (\ver(\mathcal P_1) + \ver(\mathcal P_2) - 3) - 3 \cr & = & (\ver(\mathcal P_1) - 3) + (\ver(\mathcal P_2) - 3) = \tN(\mathcal P_1) + \tN(\mathcal P_2).
	\end{eqnarray*}
	Compute the modified normalized volume of the connected sum:
	\begin{equation} \label{eqn800}
		\tomega(\mathcal P_1 \# \mathcal P_2) = \frac{\vol(\mathcal P_1) + \vol(\mathcal P_2)}{\tN(\mathcal P_1) + \tN(\mathcal P_2)} = \frac{\tN(\mathcal P_1)\cdot\tomega(\mathcal P_1) + \tN(\mathcal P_2)\cdot\tomega(\mathcal P_2)}{\tN(\mathcal P_1) + \tN(\mathcal P_2)}.
	\end{equation} 
		Thus, the value $\tomega$ for the connected sum of polyhedra behaves as the weighted average of the values $\tomega$ of the summands, with the modified numbers of vertices serving as weights.

	Consider the sequence of polyhedra $\#_k \mathcal P$ obtained by successively gluing $k \geq 1$ copies of the same polyhedron $\mathcal P$.
		Clearly, $\tomega(\#_k \mathcal P) = \tomega(\mathcal P)$.
		Observe that
	$$
	\omega(\#_k \mathcal P) = \frac{\vol(\#_k \mathcal P)}{\ver(\#_k \mathcal P)} = \frac{\vol(\#_k \mathcal P)}{\tN(\#_k \mathcal P) + 3} = \tomega(\#_k \mathcal P) \cdot \frac{\tN(\#_k \mathcal P)}{\tN(\#_k \mathcal P) + 3}.
	$$
	Since $\tN(\#_k \mathcal P) = k \cdot \tN(\mathcal P) \to \infty$ as $k \to \infty$, the fraction on the right tends to $1$.
	Thus,
	$$
	\lim_{k \to \infty} \omega(\#_k \mathcal P) = \tomega(\#_k \mathcal P ) =  \tomega(\mathcal P).
	$$
		Therefore, in order to approximate an arbitrary value in the interval $\left[ \frac{1}{4} \voct, \frac{1}{2} \voct \right]$ by normalized volumes, it suffices to approximate it by modified normalized volumes.

		Suppose that, for polyhedra $\mathcal P_1, \mathcal P_2 \in \mathcal R_{ideal}$, the modified volumes are $\tomega(\mathcal P_1) = A$ and $\tomega(\mathcal P_2) = B$.
		Without loss of generality, assume that $A < B$.
		We show how to construct a sequence of polyhedra whose normalized volumes converge to an arbitrary value in the interval $(A, B)$.
	Denote by $\#_k \mathcal P_1 \#_m \mathcal P_2$ the polyhedron obtained as the connected sum of $k$ copies of $\mathcal P_1$ and $m$ copies of $\mathcal P_2$ along triangular faces.
	Successively applying formula (\ref{eqn800}), we obtain
	$$
	\tomega(\#_k \mathcal P_1 \#_m \mathcal P_2) = \frac{k \cdot \tN(\mathcal P_1) \cdot A + m \cdot \tN(\mathcal P_2) \cdot B}{k \cdot \tN(\mathcal P_1) + m \cdot \tN(\mathcal P_2)}.
	$$
	For arbitrary $\alpha \in (0,1)$ choose sequences of integers $\{k_i\}$ and $\{m_i\}$ such that $(k_i, m_i)=1$ and 
	$$
	\lim_{i \to \infty}  \frac{k_i}{ m_i}   = \frac{\alpha \cdot \tN(\mathcal P_2)}{(1-\alpha) \cdot \tN(\mathcal P_1)}.
	$$ 
	Then 
	$$
	\lim_{i \to \infty} \tomega(\#_{k_i} \mathcal P_1 \#_{m_i} \mathcal P_2) = \alpha \cdot A + (1-\alpha) \cdot B.
	$$
	Consequently, the set of values of $\tomega$ is dense in the interval between any two values $\tomega(\mathcal P_1)$ and $\tomega(\mathcal P_2)$.

	To prove density of normalized volumes on the interval $\left[ \frac{1}{4} \voct, \frac{1}{2} \voct \right]$, we use ideal right-angled $n$-antiprisms $\mathcal A(n)$.
		Using formula~\eqref{eqn2}, we conclude that, as $n \to \infty$,
	$$
	\vol(\mathcal A(n)) \sim \frac{\voct}{2} n.
	$$
	At the same time, $\tN(\mathcal A(n)) = 2n - 3 \sim 2n$.
		Therefore,
	$$
	\lim_{n \to \infty} \tomega(\mathcal A(n)) = \lim_{n \to \infty} \frac{\frac{\voct}{2} n}{2n} = \frac{\voct}{4}.
	$$
	Thus, first, there exists an antiprism $\mathcal A(n)$ with the value $\tomega(\mathcal A(n))$ arbitrarily close to $\frac{1}{4} \voct$.
		Second, according to Theorem~\ref{theorem2.3}, there exists a sequence of polyhedra $\{Q_i\}$ for which $\lim_{i\to \infty} \frac{\vol(Q_i)}{\ver(Q_i)} = \frac{1}{2} \voct$ and hence $\lim_{i\to \infty} \tomega(Q_i) = \frac{1}{2} \voct$.

		For arbitrarily small $\varepsilon$, choose as $\mathcal P_{1, \varepsilon}$ the polyhedron $\mathcal A(n_{\varepsilon})$ for which $|
	\tomega(\mathcal A(n_{\varepsilon})) - \frac{1}{4}\voct | < \varepsilon$, and as $\mathcal P_{2, \varepsilon}$ choose the polyhedron $Q_{i_\varepsilon}$ from Theorem~\ref{theorem2.3} for which $|
		\frac{1}{2} \voct - \tomega(Q_{i_\varepsilon})| < \varepsilon$. Applying to the polyhedra $\mathcal P_{1, \varepsilon}$ and $\mathcal P_{2, \varepsilon}$ the same arguments as those applied above to the polyhedra $\mathcal P_1$ and $\mathcal P_2$, we conclude that the spectrum of normalized volumes is dense on the interval $\left(\frac{1}{4}\voct + \varepsilon, \frac{1}{2} \voct - \varepsilon \right)$.
		Thus, the spectrum of normalized volumes is dense on the interval $\left[ \frac{1}{4} \voct, \frac{1}{2} \voct \right]$. The proof of Theorem~\ref{theorem1} is complete.
\end{proof}
	
\section{Proof of Theorem~\ref{theorem2}} \label{main-compact}

\begin{proof}[Proof of Theorem~\ref{theorem2}]	
	We show that $\Omega(\mathcal R_{comp}) \subset \left[ \frac{5}{192} \voct, \frac{5}{8} \vtet \right)$.
	Dividing all parts of inequality \eqref{eqn7} by $\ver(\mathcal P)$, we obtain the following inequality for the normalized volume:
	$$
	\frac{\voct}{32} \frac{\ver(\mathcal P)-8}{\ver(\mathcal P)} \leq \omega(\mathcal P) < \frac{5\vtet}{8} \frac{\ver(\mathcal P)-10}{\ver(\mathcal P)}.
	$$
	As follows from formula~(\ref{eqn4}), the minimum number of vertices for a polyhedron from $\mathcal{R}_{comp}$ is equal to $20$.
	For $N \geq 20$, the function $h(N) = \frac{N-8}{N} = 1 - \frac{8}{N}$ increases monotonically.
	The minimum values of normalized volumes of polyhedra from $\mathcal{R}_{comp}$ with numbers of vertices from $20$ to $46$ are presented in Table~\ref{table2}.
	Substituting $\ver(\mathcal P)=48$, we obtain
	$$
	\omega(\mathcal P) \geq \frac{\voct}{32} \cdot \frac{48-8}{48} = \frac{5}{192} \voct.
	$$
	As can be seen from Table~\ref{table2}, the value $\frac{5}{192} \voct$ does not exceed the minimum values of normalized volumes of polyhedra with numbers of vertices from $20$ to $46$; therefore, $\frac{5}{192} \voct$ can be taken as an estimate for the lower bound of the spectrum.
	For the upper bound, consider the function $u(N) = \frac{5\vtet}{8} (1 - \frac{10}{N})$. It tends to $\frac{5\vtet}{8}$ as $N \to \infty$.
	Since, by Theorem~\ref{theorem3.2}, the upper estimate in~\eqref{eqn7} is asymptotically sharp, the upper bound of the spectrum is equal to $\frac{5}{8} \vtet$.

	We show that the spectrum is discrete on the interval $\left[ \frac{5}{192} \voct, \frac{1}{32} \voct \right)$.
	Indeed, given $\frac{5}{192} \voct < C < \frac{\voct}{32}$, consider the interval $[\frac{5}{192} \voct, C]$.
	Assume that $\omega(\mathcal P) \in [ \frac{5}{192} \voct, C]$. Then the lower estimate in inequality~\eqref{eqn7} gives
	$$
	\frac{\voct}{32} \left( 1 - \frac{8}{\ver(\mathcal P)} \right) \leq C \quad \text{and hence} \qquad \ver(\mathcal P) \leq \frac{8 \voct}{\voct - 32C} .
	$$
	By the choice of $C$, the expression on the right-hand side of the inequality is positive.
	Thus, the number of vertices $\ver(\mathcal P)$ is bounded above by a quantity depending only on $C$ and the constant $\voct$.

	According to Theorem~\ref{theorem2.1}, a compact hyperbolic polyhedron is determined uniquely up to isometry by its 1-skeleton, which is a 3-connected planar 3-valent graph.
	Since the number of such graphs with a given number of vertices is finite, the set of possible volume values, and hence of possible normalized volume values, is also finite.
	Thus, every interval $\left[ \frac{5}{192} \voct, C \right]$ contains only finitely many values of the spectrum.
	Consequently, the spectrum is discrete on the interval $\left[ \frac{5}{192} \voct, \frac{1}{32} \voct \right)$.

	We show that the spectrum is dense in the interval $\left[\frac{5}{16} \vtet, \frac{5}{8} \vtet \right]$.
	Define the operation of a \textit{connected sum}, or gluing, of two polyhedra from $\mathcal{R}_{comp}$ along two isometric pentagonal faces.

	For a polyhedron $\mathcal P \in \mathcal R_{comp}$, define the modified number of vertices $\tN(\mathcal P) = \ver(\mathcal P) - 10$, and also define the modified normalized volume $\tomega(\mathcal P)$ by setting  
	$$
	\tomega(\mathcal P) = \frac{\vol(\mathcal P)}{\tN(\mathcal P)}.
	$$
	The relation between the usual normalized volume $\omega(\mathcal{P})$ and the modified normalized volume $\tomega(\mathcal{P})$ is given by
	$$ 
	\tomega(\mathcal P) = \frac{\vol(\mathcal P)}{\ver(\mathcal P) - 10} = \omega(\mathcal P) \cdot \frac{\ver(\mathcal P)}{\ver(\mathcal P) - 10}. 
	$$
	Let $\mathcal P_1, \mathcal P_2 \in \mathcal R_{comp}$. Suppose that there exist pentagonal faces $F_1 \subset \mathcal P_1$ and $F_2 \subset \mathcal P_2$ isometric to each other.
	Glue the polyhedra $\mathcal P_1$ and $\mathcal P_2$ along these faces and denote the resulting polyhedron by $\mathcal P_1 \# \mathcal P_2$.
		We will call this operation a \textit{connected sum along pentagonal faces}.
	
	\begin{lemma} \label{lemma5.1}
		Under the operation of connected sum of compact right-angled polyhedra along pentagonal faces, volume behaves additively: 
		$$ 
		\vol (\mathcal P_1 \# \mathcal P_2) = \vol (\mathcal P_1) + \vol (\mathcal P_2), 
		$$ 
		and the number of vertices changes as follows: 
		$$ 
		\ver(\mathcal P_1 \# \mathcal P_2) = \ver(\mathcal P_1) + \ver(\mathcal P_2) - 10. 
		$$
		Moreover, the polyhedron $\mathcal P_1 \# \mathcal P_2$ is compact right-angled.
	\end{lemma}
\begin{proof}
		When right-angled polyhedra are glued, the faces adjacent to $F_1$ and $F_2$ pairwise form flat dihedral angles of size $\pi$ and merge into complete geodesic hyperfaces. The right angles with the remaining faces are preserved; hence $\mathcal P_1 \# \mathcal P_2$ is again a right-angled polyhedron.
		In a compact right-angled polyhedron, all vertices are 3-valent. Each vertex of a gluing face is formed by the intersection of this face itself and two adjacent faces perpendicular to it. Upon gluing, the face goes into the interior, while the two adjacent faces merge smoothly. Consequently, the point of their intersection becomes an interior point of an edge and ceases to be a vertex. Since the two glued pentagonal faces contain a total of $10$ vertices, all of them disappear from the 1-skeleton under identification, which gives the desired formula. The lemma is proved.
\end{proof}

	At the same time, the modified number of vertices behaves additively:  
	$$ 
	\tN(\mathcal P_1 \# \mathcal P_2) = (\ver(\mathcal P_1) + \ver(\mathcal P_2) - 10) - 10 = \tN(\mathcal P_1) + \tN(\mathcal P_2).
	$$
		Consequently, the modified normalized volume of a connected sum along pentagonal faces is a weighted average:
	\begin{equation}
		\tomega(\mathcal P_1 \# \mathcal P_2) = \frac{\tN(\mathcal P_1) \cdot \tomega(\mathcal P_1) + \tN(\mathcal P_2) \cdot \tomega(\mathcal P_2)}{\tN(\mathcal P_1) + \tN(\mathcal P_2)}.
	\end{equation}
	
		Consider the successive gluing of $m$ copies of the same polyhedron $\mathcal P$. Denote the resulting polyhedron by $\#_m \mathcal P$.
	Then $\tomega(\#_m \mathcal P) = \tomega(\mathcal P)$, since this is the average of identical quantities.
		The relation between the usual and modified normalized volumes is described by the formula
	$$
	\omega(\#_m \mathcal P) = \frac{\vol(\#_m \mathcal P)}{\ver(\#_m \mathcal P)} = \frac{\vol(\#_m \mathcal P)}{\tN(\#_m \mathcal P) + 10} = \tomega(\#_m \mathcal P ) \frac{\tN(\#_m \mathcal P)}{\tN(\#_m \mathcal P)+10}.
	$$
	Since $\tN(\#_m \mathcal  P) = m \cdot \tN(\mathcal  P) \to \infty$ as $m \to \infty$, the fraction on the right tends to $1$.
	Consequently,
	$$ 
	\lim_{m \to \infty} \omega(\#_m \mathcal P) = \tomega(\#_m \mathcal P) = \tomega(\mathcal P).
	$$
		This means that every value of the modified normalized volume $\tomega(\mathcal P)$ is a limit point for the spectrum of the usual normalized volumes $\omega$.
		Thus, in order to approximate an arbitrary value in the interval $[\frac{5\vtet}{16}, \frac{5\vtet}{8}]$ by normalized volumes, it suffices to approximate it by modified normalized volumes.

	Let $\mathcal P_1$ and $\mathcal P_2$ be two polyhedra from $\mathcal R_{comp}$.
	Let $\tomega(\mathcal P_1) = A$ and $\tomega(\mathcal P_2) = B$, where $A<B$.
		Consider the polyhedron $\#_k \mathcal P_1 \#_{\ell} \mathcal P_2$, which is the connected sum of $k$ copies of $\mathcal P_1$ and $\ell$ copies of $\mathcal P_2$ along pairwise isometric pentagonal faces. Then
	$$
	\tomega(\#_k \mathcal P_1 \#_{\ell} \mathcal P_2) = \frac{k \cdot \tN(\mathcal  P_1) \cdot A + \ell \cdot \tN(\mathcal P_2) \cdot B}{k \cdot \tN(\mathcal P_1) + \ell \cdot \tN(\mathcal P_2)}.
	$$
	
		For arbitrary $\alpha \in (0,1)$, choose sequences of integers $\{k_i\}$ and $\{ \ell_i\}$ such that
	$$
	\lim_{i\to \infty} \frac{k_i}{\ell_i} = \frac{\alpha \cdot \tN(\mathcal P_2)}{(1-\alpha) \cdot \tN(\mathcal P_1)}.
	$$
	Then 
	$$
	\lim_{i \to \infty} \tomega(\#_{k_i} \mathcal P_1 \#_{\ell_i} \mathcal P_2) = \alpha \cdot A + (1-\alpha) \cdot B.
	$$
	Consequently, the set of values of $\tomega$ is dense in the interval between any two values $\tomega(\mathcal P_1)$ and $\tomega(\mathcal P_2)$.

	Now we pass to concrete families of polyhedra.
		To realize a connected sum, we need polyhedra with isometric pentagonal faces.
		Consider the L\"obell polyhedra $\mathcal L(n)$ and their towers $\mathcal L_k(n)$.
		For fixed $n$, all lateral pentagonal faces of these polyhedra are isometric to each other. Denote the set of these pentagonal faces by $\mathcal{F}_n$.
		The polyhedra $\mathcal L(n)$ and $\mathcal L_k(n)$ have $2n$ pairwise isometric pentagonal faces each.

		Moreover, observe that if arbitrary polyhedra $\mathcal P_1$ and $\mathcal P_2$ have, respectively, $p_1$ and $p_2$ pairwise isometric pentagonal faces from the set $\mathcal{F}_n$, then the polyhedron $\mathcal P_1 \# \mathcal P_2$ obtained as the connected sum along a pair of such faces has at least $p_1 + p_2 - 12$ isometric pentagonal faces of the class $\mathcal{F}_n$. Indeed, under gluing, the two faces $F_1 \subset \mathcal P_1$ and $F_2 \subset \mathcal P_2$ go into the interior of the polyhedron. In addition, the faces incident to the glued faces may change their combinatorial type, and thereby cease to belong to the class $\mathcal{F}_n$; there are exactly $5$ such faces in each of the polyhedra. In the worst case, all these $12$ faces belonged to $\mathcal{F}_n$ and were lost.

		Below we assume that $n>6$. Then, after several connected-sum operations applied to copies of the polyhedra $\mathcal L(n)$ and $\mathcal L_k(n)$ along pentagonal faces from the set $\mathcal{F}_n$, the resulting polyhedron will have no fewer pentagonal faces of the class $\mathcal{F}_n$ than $\mathcal L(n)$ and $\mathcal L_k(n)$ have.
		This makes it possible to repeat the connected-sum operation on copies of the polyhedra $\mathcal L(n)$ and $\mathcal L_k(n)$ as many times as required.

		First, note that, by formula~(\ref{eqn6}), the endpoints of the interval of interest $\left[\frac{5}{16} \vtet, \frac{5}{8} \vtet \right]$ can be obtained as limits of sequences of modified volumes of the polyhedra $\mathcal L(n)$ (as $n$ tends to infinity) and $\mathcal L_k(n)$ (as $n$ and $k$ tend to infinity).
	
		To complete the proof, observe that, by formula~(\ref{eqn6}), for arbitrarily small $\varepsilon > 0$ there exist natural numbers $k_{\varepsilon}$ and $n_{\varepsilon}$ such that
	$$
	\tomega(\mathcal L(n_{\varepsilon}))  \in [\tfrac{5\vtet}{16}, \tfrac{5\vtet}{16} + \varepsilon], \qquad
	\tomega(\mathcal L_{k_{\varepsilon}}(n_{\varepsilon}))  \in [\tfrac{5\vtet}{8} - \varepsilon, \tfrac{5\vtet}{8}].
	$$ 
		Thus, we can construct connected sums of the polyhedra $\mathcal L(n_{\varepsilon})$ and $\mathcal L_{k_{\varepsilon}}(n_{\varepsilon})$ of the form
	$$
	\#_m \mathcal L(n_{\varepsilon}) \#_s \mathcal L_{k_{\varepsilon}}(n_{\varepsilon})
	$$
		and apply the procedure described above to approximate every number in the interval $[\tfrac{5\vtet}{16} + \varepsilon, \tfrac{5\vtet}{8} - \varepsilon]$.
	The proof of Theorem~\ref{theorem2} is complete. 
\end{proof}
		
\section{Comparison of spectra and open questions} \label{comparison}

As we mentioned in the Introduction, there is a close relation between volumes of fully augmented links and volumes of ideal right-angled polyhedra. It is known (see~\cite{La04, Pur2011, VeEg24}) that the complement of every fully augmented hyperbolic link $L \subset S^3$ can be cut into two ideal right-angled polyhedra $\mathcal P$ isometric to each other.
In this case, the following relations hold:
\begin{equation} \label{eqn900}
\vol(L) = 2\vol(\mathcal P), \qquad \ver(\mathcal P) = 3  a(L),
\end{equation}
where $a(L)$ is the number of augmentations (vertical components) of the link.
	
Kwon and Tham~\cite{KwonTham} studied the spectrum of the volume density $\operatorname{vd}(L)$ of a fully augmented link $L$, which is defined as follows: $\operatorname{vd}(L) = \vol(L)/a(L)$.
We express this quantity in terms of the normalized volume of the polyhedron $\mathcal P$ corresponding to the link $L$, using relations~(\ref{eqn900}):
$$
\operatorname{vd}(L) = \frac{2\vol(\mathcal P)}{ \ver(\mathcal P)/3} = 6 \frac{\vol(\mathcal P)}{\ver(\mathcal P)} = 6 \, \omega(\mathcal P).
$$
It was shown in~\cite{KwonTham} that the spectrum of $\operatorname{vd}(L)$ is contained in the interval $\left[ \voct, 10\vtet \right]$, is discrete in $\left[ \voct, 2\voct \right)$, and is dense in $\left[ 2\voct, 10\vtet \right]$. Recalculating these intervals for $\omega(\mathcal P)$ by division by $6$, we conclude that, for the subclass of polyhedra corresponding to fully augmented links, the spectrum lies in the interval $\left[\frac{1}{6} \voct, \frac{5}{3} \vtet \right]$, is discrete in $\left[ \frac{1}{6} \voct, \frac{1}{3} \voct \right)$, and is dense in $\left[ \frac{1}{3} \voct, \frac{5}{3} \vtet \right]$.
	
Comparing the spectrum $\Omega(\textrm{FAL})$ of volume densities of augmented links with the spectrum $\Omega(\mathcal R_{ideal})$ of normalized volumes of ideal right-angled polyhedra, we note the following facts (Fig.~\ref{fig10}).
\begin{itemize}
\item[(1)] The lower bounds of the spectra coincide and are equal to $d = \frac{1}{6} \voct$.
\item[(2)] The upper bound of the spectrum $\Omega(\mathcal R_{ideal})$ is equal to $k=\frac{1}{2} \voct$, which exceeds the upper bound $h= \frac{5}{3}\vtet$ of the spectrum $\Omega(\textrm{FAL})$. This difference is explained by the fact that the class of all ideal right-angled polyhedra is substantially broader than the class of polyhedra arising from link diagrams. For example, in polyhedra arising from links, every vertex is incident to two triangles which intersect exactly at this vertex.
\item[(3)] The structures of the spectra qualitatively coincide: first comes a discrete part and then a dense part. The transition from the discrete part to the dense part in the case of $\Omega(\mathcal R_{ideal})$ occurs at the value $f = \frac{1}{4} \voct$, whereas in the case of $\Omega(\textrm{FAL})$ it occurs at the value $g = \frac{1}{3} \voct$.
\item[(4)] As can be seen from the proof of Theorem~\ref{theorem1} and Table~\ref{table1}, the only normalized volume from $\Omega(\mathcal R_{ideal})$ that falls into the dense part of the spectrum $\Omega(\mathcal R_{comp})$ is $d = \frac{1}{6} \voct = 0.610643 \in [d, e)$, corresponding to the octahedron.
\end{itemize}	
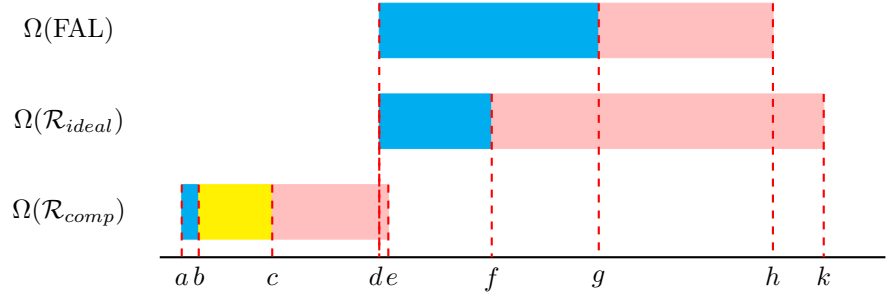
\begin{figure}[ht]
    \centering
    \setlength{\unitlength}{1.0mm}
    \begin{tikzpicture}[scale=1.2] 
    \draw[thick, black] (0,0) -- (8,0); 
\coordinate (A1) at (0.2424, 0.2) ;
\coordinate (B1) at (0.2424, 0.8) ;
\coordinate (C1) at (0.43, 0.2) ;
\coordinate (D1) at (0.43, 0.8) ;
\coordinate (E1) at (1.24, 0.2) ;
 \coordinate (F1) at (1.24, 0.8) ;
  \coordinate (G1) at (2.52, 0.2) ;
    \coordinate (H1) at (2.52, 0.8) ; 
    \filldraw[cyan]  (A1) -- (B1) -- (D1) -- (C1) -- cycle;
    \filldraw[yellow]  (C1) -- (D1) -- (F1) -- (E1) -- cycle;
   \filldraw[pink]  (E1) -- (F1) -- (H1) -- (G1) -- cycle;        
\draw[thick, dashed, red] (0.2424,0.8) -- (0.2424,0); 
\draw[thick, dashed, red] (0.43,0.8) -- (0.43,0); 
\draw[thick, dashed, red] (1.24,0.8) -- (1.24,0); 
\draw[thick, dashed, red] (2.52,0.8) -- (2.52,0); 
\coordinate (A2) at (2.42, 1.2) ;
\coordinate (B2) at (2.42, 1.8) ;
\coordinate (C2) at (3.66, 1.2) ;
\coordinate (D2) at (3.66, 1.8) ;
\coordinate (E2) at (7.32, 1.2) ;
 \coordinate (F2) at (7.32, 1.8) ;
    \filldraw[cyan]  (A2) -- (B2) -- (D2) -- (C2) -- cycle;
    \filldraw[pink]  (C2) -- (D2) -- (F2) -- (E2) -- cycle;
\draw[thick, dashed, red ] (7.32,1.8) -- (7.32,0); 
\draw[thick, dashed, red] (3.66,1.8) -- (3.66,0); 
\draw[thick, dashed, red] (2.42,1.8) -- (2.42,0); 
\coordinate (A3) at (2.42, 2.2) ;
\coordinate (B3) at (2.42, 2.8) ;
\coordinate (C3) at (4.84, 2.2) ;
\coordinate (D3) at (4.84, 2.8) ;
\coordinate (E3) at (6.76, 2.2) ;
 \coordinate (F3) at (6.76, 2.8) ;
    \filldraw[cyan]  (A3) -- (B3) -- (D3) -- (C3) -- cycle;
    \filldraw[pink]  (C3) -- (D3) -- (F3) -- (E3) -- cycle;
\draw[thick, dashed, red] (6.76,2.8) -- (6.76,0); 
\draw[thick, dashed, red] (4.84,2.8) -- (4.84,0); 
\draw[thick, dashed, red] (2.42,2.8) -- (2.42,0); 
\draw(0.2424,-0.24) node {$a$};
\draw(0.43,-0.2) node {$b$};
\draw(1.24,-0.24) node {$c$};
\draw(2.38,-0.2) node {$d$};
\draw(2.56,-0.24) node {$e$};
\draw(3.66,-0.24) node {$f$};
\draw(4.84,-0.24) node {$g$};
\draw(6.76,-0.2) node {$h$};
\draw(7.32,-0.2) node {$k$};
   \draw(-1,0.5) node {$\Omega(\mathcal R_{comp})$};
      \draw(-1,1.5) node {$\Omega(\mathcal R_{ideal})$};
         \draw(-1,2.5) node {$\Omega(\textrm{FAL})$};
    \end{tikzpicture}
    \caption{Comparison of spectra of normalized volumes.}
    \label{fig10}   
\end{figure}
The corresponding intervals are shown in Fig.~\ref{fig10}, where the indicated quantities have the following values: 
$$
\begin{gathered}
a = \frac{5}{192} \voct=0.095413, \qquad b = \frac{1}{32}\voct=0.114495, \qquad c=\frac{5}{16}\vtet = 0.317169, \cr 
d = \frac{1}{6} \voct = 0.610643, \qquad e=\frac{5}{8} \vtet = 0.634338, \qquad f = \frac{1}{4} \voct = 0.915965, \cr 
g = \frac{1}{3} \voct = 1.221287, \qquad h = \frac{5}{3} \vtet = 1.691569, \qquad k=\frac{1}{2} \voct = 1.831931.
\end{gathered}  
$$	
In the color version of Fig.~\ref{fig10}, the regions of the discrete part of the spectrum are marked in blue, the regions of the dense part of the spectrum in pink, and the region that remains unexplored in yellow.
		
In conclusion, we formulate several open questions and one conjecture.
	
\begin{question} \label{question2}
{\rm 
What is the exact lower bound of the spectrum $\Omega(\mathcal R_{comp})$ of normalized volumes of compact right-angled hyperbolic polyhedra? The bound $a = \frac{5}{192} \voct$ indicated in the present paper follows from formula~(\ref{eqn7}) and Table~\ref{table2}, but it is not sharp.
}
\end{question}
	
\begin{question} \label{question3}
{\rm What is the structure of the spectrum $\Omega(\mathcal R_{comp})$ of normalized volumes of compact right-angled hyperbolic polyhedra on the interval $\left[ b, c \right] = \left[ \frac{1}{32} \voct, \frac{5}{16} \vtet \right]$?}
\end{question}

Our conjecture is as follows. 
\begin{conjecture}
{\rm The spectrum $\Omega(\mathcal R_{comp})$ of normalized volumes of compact right-angled polyhedra lies in the interval $\left[ \frac{\vol(\mathcal L(5))}{20}, \frac{5}{8} \vtet \right]$, with both bounds sharp. Moreover, the spectrum $\Omega(\mathcal R_{comp})$ is discrete in the interval $\left[ \frac{\vol(\mathcal L(5))}{20}, \frac{5}{16} \vtet \right]$ and dense in $\left[ \frac{5}{16} \vtet, \frac{5}{8} \vtet \right]$.}
\end{conjecture}

\begin{question} \label{question4}
 {\rm What is the structure of the spectrum of normalized volumes of right-angled hyperbolic polyhedra admitting both finite and ideal vertices? It was shown in~\cite{VeEg25} that the smallest volume in this class of polyhedra is equal to $\frac{1}{4} \voct$ and is attained on the triangular bipyramid, which has two finite and three ideal vertices.}
\end{question}

Comparing the volume of the antiprism $\mathcal A(4)$ and the volume of the L\"obell polyhedron $\mathcal L(6)$, one can observe that, to six decimal places, both volumes are equal to $6.023046$. Moreover, the volumes of $\mathcal A(4)$ and $\mathcal L(6)$ computed to 50 digits after the decimal point again coincide and are equal to
$$
6.02304602004718882363418931461679711549802902472249\ldots
$$
Is it true that the equality $\vol (\mathcal A(4)) = \vol (\mathcal L(6))$ holds? Since both volumes are expressed in terms of the Lobachevsky function, this question reduces to the question of whether the following equality holds:
\begin{equation} \label{eqn8}
3 \left[ 2 \Lambda (\theta_6)  +  \Lambda \left( \theta_6 + \frac{\pi}{6} \right)  +  \Lambda \left( \theta_6 - \frac{\pi}{6} \right)  +  \Lambda \left( \frac{\pi}{2} - 2 \theta_6 \right) \right] 
= 8 \Lambda \left( \frac{3 \pi}{8} \right) + 8  \Lambda \left(\frac{\pi}{8}\right),   
\end{equation} 
where $\theta_6  =  \frac{\pi}{2}  -  \arccos  \left(\frac{1}{\sqrt{3}} \right)$.
Observe that the polyhedron assembled from two ideal right-angled antiprisms $\mathcal A(4)$ is known as the cuboctahedron $Q$ (see, for instance, Fig.~3 in~\cite{VeEg25}). A formula for $\vol(Q)$ was obtained in~\cite[Lem.~5.5]{Ad26} by decomposing $Q$ into ideal tetrahedra. It follows from this formula that, for $\phi=\arctan(\sqrt{2})$,
\begin{equation} \label{eqn9}
4  \Lambda \left(\frac{\pi}{2} - \phi \right)  + 8 \Lambda \left( \phi \right)  - 3  \Lambda \left( 2 \phi \right)  +  \frac{1}{2} \Lambda \left( 4 \phi \right) 
= 8 \Lambda \left( \frac{3 \pi}{8} \right) + 8  \Lambda \left(\frac{\pi}{8}\right),   
\end{equation} 

\begin{question} \label{question1}
{\rm 
Does equality (\ref{eqn8}) hold? In particular, does it follow from~(\ref{eqn9})? 
}
\end{question}

\end{document}